\documentclass[12pt]{amsart}
\usepackage{enumerate}
\usepackage{enumitem}
\usepackage{graphicx,graphics}
\usepackage{soul}
\usepackage{amsfonts}
\usepackage{amssymb}
\usepackage{amsthm}
\usepackage{amsmath}
\usepackage{marginnote}
\usepackage{mathrsfs,color}
\usepackage{tikz} % diagramas
\usepackage{hyperref}
\usepackage{verbatim}
\input{xy}
\xyoption{all}

\numberwithin{equation}{section}

\newtheorem{theorem}{Theorem}[section]
\newtheorem{corollary}{Corollary}[theorem]
\newtheorem{lemma}[theorem]{Lemma}
\newtheorem{proposition}[theorem]{Proposition}
\theoremstyle{definition}
\newtheorem{definition}[theorem]{Definition}
\theoremstyle{definition}
\newtheorem{problem}[theorem]{Problem}
\theoremstyle{definition}
\newtheorem{example}[theorem]{Example}
\theoremstyle{definition}
\newtheorem{remark}[theorem]{Remark}
\theoremstyle{theorem}
\newtheorem{claim}{Claim}
\theoremstyle{definition}

\begin{document}

\title[Decreasing involutions and $\mathbb{Z}_{2}$-means]{Decreasing involutions and $\mathbb{Z}_{2}$-means}

\author{Natalia Jonard-P\'erez} \email{nat@ciencias.unam.mx}

\author{Ananda L\'opez-Poo} \email{anandalc95@gmail.com}

\address{Departamento de  Matem\'aticas,
Facultad de Ciencias, Universidad Nacional Aut\'onoma de M\'exico, 04510 Ciudad de M\'exico, M\'exico.}

\keywords{Equivariant  retract, $n$-mean, involutions, duality, group actions, topological lattices}

\subjclass[2020]{Primary: 
%06D50, %Lattices and duality
22A26,  	%Topological semilattices, lattices and applications
26E60,  %Means
%54C05, Continuous maps 
%54C99, % Maps and general types of topological spaces defined by maps none of the above
 54C55, %Absolute neighborhood extensor, absolute extensor, absolute neighborhood retract (ANR), absolute retract spaces (general properties) 
  54H15, %Transformation groups and semigroups (topological aspects)  
Secondary: 
26B25, %	Convexity of real functions of several variables, generalizations
52A20 %Convex sets in $n$ dimensions (including convex hypersurfaces)
54H12,  	%Topological lattices, etc. (topological aspects)
	%52A05,  	%Convex sets without dimension restrictions (aspects of convex geometry)
    %54B20  %	Hyperspaces in general topology}
    %91B14, % Social Choice    
%01A60, %History of mathematics in the 20th century 
%55P20, %Eilenberg-Mac Lane spaces
}
\thanks{This work has been supported by  PAPIIT grant IN104525 (UNAM, M\'exico). %Also, the second author has been supported by Conahcyt grant 712523.
}

\begin{abstract} 

An \(n\)-mean on a topological space \(X\) is a continuous symmetric map
\(p\colon X^n\to X\) satisfying \(p(x,\ldots,x)=x\) for every
\(x\in X\). If \(X\) is a \(G\)-space, such an \(n\)-mean is equivariant
if $p(gx_1,\ldots,gx_n)=g p(x_1,\ldots,x_n)$
for every \(g\in G\) and \(x_1,\ldots,x_n\in X\). For finite \(G\), we
prove that, in the presence of an equivariant \(n\)-mean with
 \(n\) a multiple of $\lvert G\rvert$, the $\mathrm{ANE}$, $\mathrm{AE}$, and $\mathrm{AR}$ properties imply their
equivariant counterparts.

We next consider \(\mathbb Z_2\)-actions induced by decreasing
involutions on topological lattices. We first prove that every fixed
point of such an involution on a modular lattice yields an explicit
equivariant \(2\)-mean. Furthermore, we introduce a lattice-theoretic
version of the Babylonian iteration used to approximate the geometric
mean and prove that, under suitable conditions on the order and its
interaction with the topology, this iteration produces an equivariant
\(2\)-mean whenever the lattice admits a \(2\)-mean compatible with the
order. Finally, we apply these results
to spaces of compact convex bodies and supercoercive convex functions,
where we prove the existence of equivariant \(2\)-means and show that
both spaces are \(\mathbb Z_2\)-absolute retracts.

\end{abstract}

\maketitle 

\section{Introduction}

One of the longstanding open problems in equivariant topology is
Jaworowski's problem, which asks whether the property of being a
\(G\)-absolute retract can be characterized in terms of fixed-point
sets. Indeed, if \(G\) is a compact Lie group and \(X\) is a metrizable
\(G\)-$\mathrm{AR}$, then \(X^{H}\) is an $\mathrm{AR}$ for every closed subgroup \(H\) of
\(G\); this necessary condition follows from a result of Smirnov
\cite{smirnov}. In his work on equivariant extensions and $G$-retracts
\cite{jaworowski1973,jaworowski1976,jaworowski1980}, Jaworowski asked
whether this necessary condition is also sufficient when the action has
only finitely many orbit types.

\begin{problem}[Jaworowski's problem]
\label{problema de Jaworowski}
Let \(G\) be a compact Lie group and let \(X\) be a metrizable
\(G\)-space with finitely many orbit types. Suppose that \(X^{H}\) is
an $\mathrm{AR}$ for every closed subgroup \(H\) of \(G\). Must \(X\) be a
\(G\)-$\mathrm{AR}$?
\end{problem}

Several partial results have been obtained under additional assumptions
(see, for instance, \cite{Antonyan2005,AntonyanEtAl2017}; for a recent
discussion of the problem, see also \cite{sergey2022}). Nevertheless,
Jaworowski's problem remains open even when $G$ is finite\footnote{When $G$ is finite, $X$ automatically has only finitely many orbit
types.}. More specifically, the problem remains open even for
the simplest nontrivial group, \(G=\mathbb Z_2\).

A particularly significant special case arises when
\(G=\mathbb Z_2\), \(X\) is homeomorphic to the Hilbert cube
\[
Q=\prod_{n=1}^{\infty}[-1,1],
\]
and \(X^{\mathbb Z_2}\) consists of a single point. In this setting,
the action is determined by an involution with a unique fixed point,
and Jaworowski's problem is equivalent to another classical open
problem in infinite-dimensional topology, posed by R.~D. Anderson in the mid-sixties
(\cite[Problem~930]{WestOpenProblems}).

\begin{problem}[Anderson's problem ]
\label{problema de Anderson}
Let
\[
\sigma:Q\longrightarrow Q,\qquad
\sigma\big((x_n)_{n=1}^{\infty}\big)=(-x_n)_{n=1}^{\infty},
\]
be the standard involution on the Hilbert cube. If
\(\alpha:Q\rightarrow Q\) is an involution with a unique fixed point,
must \(\alpha\) be conjugate to \(\sigma\)? Equivalently, does there
exist a homeomorphism \(h:Q\rightarrow Q\) such that
\[
h\circ\alpha=\sigma\circ h?
\]
\end{problem}

The equivalence between Anderson's problem and this particular case of
Jaworowski's problem is explained in
\cite[Section~3.2]{sergey2022}. Despite several partial results and
related investigations
(\cite{sergey1999,west,WestWong,wong}), Anderson's problem remains open.

Another approach to Anderson's problem is to consider involutions that
are compatible with an additional order structure. Notice that the Hilbert cube
$Q$ carries a natural topological lattice structure induced by the
coordinatewise order
\[
x=(x_n)_{n=1}^{\infty}\leq y=(y_n)_{n=1}^{\infty}
\quad\Longleftrightarrow\quad
x_n\leq y_n \text{ for every } n\in\mathbb N.
\]
The corresponding meet and join operations are defined coordinatewise
by
\[
x\wedge y=(\min\{x_n,y_n\})_{n=1}^{\infty}
\quad\text{and}\quad
x\vee y=(\max\{x_n,y_n\})_{n=1}^{\infty}.
\]
Moreover, the standard involution $\sigma(x)=-x$ is decreasing with
respect to this order: if $x\leq y$, then
$\sigma(y)\leq\sigma(x)$. This observation naturally suggests the
following weaker version of Anderson's problem.

\begin{problem}[An order-theoretic version of Anderson's problem]
\label{order-theoretic-Anderson-problem}
Let $\alpha:Q\longrightarrow Q$ be an involution with a unique fixed
point. Suppose that $Q$ admits a lattice structure, with associated
partial order $\preceq$, such that $\alpha$ is decreasing with respect
to $\preceq$. Must $\alpha$ be conjugate to the standard involution
$\sigma$?
\end{problem}

A particular instance of this question was answered affirmatively in
\cite{HiguerasJonard}. Let $\mathcal K_0^n$ denote the family of all
closed convex subsets of $\mathbb R^n$ containing the origin, endowed
with the Attouch--Wets metric and ordered by inclusion. It was proved
that $(\mathcal K_0^n,d_{AW})$ is homeomorphic to the Hilbert cube and
that the polar involution is conjugate to $\sigma$. Moreover, every
decreasing involution on $\mathcal K_0^n$ having a unique fixed point
is conjugate to the polar involution and, consequently, to the standard
involution on $Q$ (\cite[Corollary~3]{HiguerasJonard}).

A key ingredient in the proof of this corollary is a representation
theorem due to B.~A. S{\l}omka, which states that every decreasing
involution
$f:\mathcal K_0^n\longrightarrow\mathcal K_0^n$ is of the form
\[
f(A)=T(A^\circ),
\]
where $T:\mathbb R^n\longrightarrow\mathbb R^n$ is a suitable
symmetric linear isomorphism \cite[Corollary~4]{Slomka}.
S{\l}omka's representation theorem is not an isolated result.
Artstein-Avidan and Milman obtained similar classification results for
several natural classes of functions
(\cite{ArtsteinAvidanMilmanDuality,
ArtsteinAvidanMilmanHidden}). Most notably, they proved that every
order-reversing involution on the class of lower semicontinuous convex
functions on $\mathbb R^n$ is, up to linear terms, the
Legendre--Fenchel transform (\cite{ArtsteinAvidanMilman}). These results
show that order-reversing involutions often exhibit a strong rigidity:
up to relatively simple modifications, they are determined by
canonical duality transforms.

Inspired by these results, the aim of this paper is to investigate a
particular case of Jaworowski's problem in which
$G=\mathbb Z_2$, the underlying space $X$ is a topological lattice,
and the action is generated by an involution that is decreasing with
respect to the lattice order. Our strategy is to use the topological
lattice structure of $X$ to study the existence of equivariant
$n$-means.

Recall that an $n$-mean on a topological space $X$ is a continuous map
$p:X^n\longrightarrow X$ which is invariant under permutations of its
arguments and, in addition, satisfies
\[
p(x,\ldots,x)=x
\]
for every $x\in X$. If $X$ is a $G$-space, the $n$-mean is
said to be equivariant if
\[
p(gx_1,\ldots,gx_n)=g\,p(x_1,\ldots,x_n)
\]
for every $g\in G$ and every $(x_1,\ldots,x_n)\in X^n$ (see
Definition~\ref{def:n-mean} for the formal definition).

The reason for considering equivariant $n$-means lies in their close
relationship with absolute retract properties. In the
non-equivariant setting, for suitable classes of connected $\mathrm{ANR}$-spaces,
the existence of an $n$-mean for some $n\geq 2$ is equivalent to the
space being an $\mathrm{AR}$ (see, for instance,
\cite[Theorem~2.2]{JuarezAnguiano2024}). This connection goes back to
the work of Eckmann and Eckmann, Ganea and Hilton (\cite{Eckmann1954, EckmannGaneaHilton}; see also
\cite{Eckmann2004,ChichilniskyHeal,Weinberger}).
Equivariant versions of these results were obtained by
Juárez-Anguiano  in \cite{JuarezAnguiano2020} and \cite{JuarezAnguiano2024}.
More recently, in \cite{JonardLopezPoo}, we obtained further
conditions under which the existence of a suitable equivariant $n$-mean
implies that a $G$-space is a $G$-$\mathrm{AR}$.

Continuing this line of research, in Section~3 we strengthen some
previous results concerning equivariant $n$-means. In particular, we
prove that if $X$ is a metrizable $\mathbb Z_2$-space whose
underlying space is an $\mathrm{AR}$, then the existence of an equivariant
$2$-mean on $X$ implies that $X$ is a $\mathbb Z_2$-$\mathrm{AR}$
(Corollary~\ref{c:equivariant mean implies Z2AR}).

With this result in mind, and in order to study Jaworowski's problem
for $\mathbb Z_2$-decreasing topological lattices, we explore two
settings in which an equivariant mean can be constructed. The first
arises when the underlying lattice is modular and the decreasing
involution has a fixed point. The second is based on a generalization
of the Babylonian method for constructing the geometric mean.

Recall that, for positive real numbers $x$ and $y$, this method can be
described by setting $a_0=x$, $h_0=y$, and defining
\[
a_{k+1}=\frac{a_k+h_k}{2},
\qquad
h_{k+1}
=\frac{2}{\dfrac{1}{a_k}+\dfrac{1}{h_k}}
=\left(
\frac{a_k^{-1}+h_k^{-1}}{2}
\right)^{-1}.
\]

The sequence $(a_k)_{k\geq 1}$ is decreasing, the sequence
$(h_k)_{k\geq 1}$ is increasing, and both converge to the geometric
mean $\sqrt{xy}$. Moreover, $a_kh_k=xy$ for every $k\geq 0$, and hence
\[
a_{k+1}
=\frac{1}{2}\left(a_k+\frac{xy}{a_k}\right).
\]
Thus, the sequence $(a_k)$ is generated by the classical Babylonian
method for approximating $\sqrt{xy}$ (see, for instance,
\cite{BartleSherbert}).

Notice also that the second recurrence is obtained from the first by
applying the multiplicative inverse to the variables and then to the
resulting arithmetic mean. This observation makes it possible to
replace the multiplicative inverse by a decreasing involution in a
more general setting.

V. Milman and L. Rotem adapted this construction to the space
$\mathcal K_{(0),b}^n$ of compact convex bodies in $\mathbb R^n$
containing the origin in their interior. More precisely, they replaced
the multiplicative inverse by the polar involution and the usual
arithmetic mean by the Minkowski arithmetic mean
\[
A(K,L)=\frac{K+L}{2}.
\]
The corresponding harmonic mean is then given by
\[
H(K,L)
=
\left(
\frac{K^\circ+L^\circ}{2}
\right)^\circ.
\]
Using these two operations, in \cite{MilmanRotem} (see also \cite{RotemAlgebraic}), they constructed a geometric mean on
$\mathcal K_{(0),b}^n$ .

Motivated by the classical Babylonian method and by the construction
of Milman and Rotem, in Section~4.2 we generalize this iterative
procedure to topological lattices equipped with a decreasing
involution. We give sufficient conditions under which the two
resulting sequences converge to a common limit. This limit defines a
generalized geometric mean which is, by construction, an equivariant
$2$-mean.

We conclude the paper with two applications that may be of interest in
convex analysis. In Section~5.1, we consider the space
$\mathcal K_{(0),b}^n$ of all compact convex bodies in $\mathbb R^n$
containing the origin in their interior, equipped with the Hausdorff
metric. We prove that, for every involution
$\alpha:\mathcal K_{(0),b}^n\longrightarrow\mathcal K_{(0),b}^n$
that is decreasing with respect to the order given by the inclusion, the $\mathbb Z_2$-space
$(\mathcal K_{(0),b}^n,\alpha)$ is a $\mathbb Z_2$-$\mathrm{AR}$.

In Section~5.2, we study the space of all convex supercoercive functions
$f:\mathbb R^n\longrightarrow\mathbb R$, that is, convex functions
satisfying
\[
\lim_{\|x\|\to\infty}\frac{f(x)}{\|x\|}=+\infty.
\]
We equip this space with the topology of epi-convergence, which in this
setting coincides with the compact-open topology. We prove that it
admits an equivariant mean with respect to Legendre--Fenchel
conjugation and, as a consequence, that the resulting
$\mathbb Z_2$-space is a $\mathbb Z_2$-$\mathrm{AR}$.

The paper is organized as follows. In Section~2, we introduce the
basic notions and preliminary results concerning group actions,
equivariant retracts, $n$-means, and topological lattices. In
Section~3, we strengthen some previous results relating equivariant
$n$-means to the $G$-$\mathrm{AR}$ property. Section~4 is devoted to decreasing
involutions on topological lattices and to the construction of
equivariant means in this setting. Finally, Section~5 contains the two
applications to convex analysis described above.

\section{Preliminaries}

Throughout the paper, all spaces are assumed to be Hausdorff, and we
use the term \textit{map} to mean a continuous function.

 We refer the reader to \cite{Bredon} for basic notions of the theory of G-spaces. However, we recall some special definitions and results that will be used throughout this text.

An \textit{action} of a topological group $G$ on a topological space $X$ is a map $\alpha:G\times X\rightarrow X$ such that $\alpha\left(e,x\right)=x$ and $\alpha\left(g,\alpha\left(h,x\right)\right)=\alpha\left(gh,x\right)$, where $g,h\in G$, $x\in X$, and $e$ is the identity element of $G$. To simplify the notation, we write $gx$ instead of $\alpha\left(g,x\right)$. A topological space $X$ equipped with an action of a Hausdorff topological group $G$ is called a \textit{$G$-space}.

Let $X$ be a $G$-space. If $A$ is a subset of $X$, the \textit{saturation} of $A$, denoted by $G\left(A\right)$, is defined as the set $\left\{ga \mid g\in G, \mbox{ } a\in A\right\}$. If $G\left(A\right)=A$, we say that $A$ is \textit{invariant}. If $A=\left\{x\right\}$ for some $x\in X$, we simply write $G\left(x\right)$ instead of $G\left(\left\{x\right\}\right)$. In this case, we say that $G\left(x\right)$ is the \textit{orbit} of $x$.

If $x\in G$, we denote by $G_{x}$ the \textit{stabilizer} of $x$, defined by $G_{x}:=\left\{g\in G \mid gx=x\right\}$. Given a subgroup $H$ of $G$, we denote by $X^{H}$ the set of points of $X$ that are fixed by $H$. Namely,  
\[
X^{H}:=\left\{x\in X \mid hx=x \mbox{ for each } h\in H\right\}.
\]
If $x\in X^{H}$, we say that $x$ is a \textit{$H$-fixed point}. It is well known that, for every subgroup $H\leq G$, the set $X^H$ is always closed. 

The family of all subgroups of $G$ that are conjugated with $H$ will be denoted by $(H)$. That is, $(H)=\left\{gHg^{-1} \mid g\in G\right\}$. For every $x\in X$, $(G_{x})$ will be called the \textit{orbit type} of the orbit $G\left(x\right)$. If $(H)$ is the orbit type of some orbit of $X$, then $(H)$ will be called an \textit{orbit type} of $X$.

A function $f:X\rightarrow Y$ between $G$-spaces is called $G$-\textit{equivariant} (or simply \textit{equivariant}) if $f\left(gx\right)=gf\left(x\right)$ for every $g\in G$, $x\in X$. If in addition $f$ is continuous, then we say that $f$ is a $G$-\textit{map}.

A $G$-space $Y$ is called a \textit{$G$-absolute neighborhood extensor} (denoted by $G$-$\mathrm{ANE}$) if for every closed and invariant subset $A$ of a metrizable $G$-space $X$ and every $G$-map $f:A\rightarrow Y$, there exist an invariant neighborhood $U$ of $A$ in $X$ and a $G$-map $F:U\rightarrow Y$ such that $F\restriction_{A}=f$. If we can always take $U=X$, then we say that $Y$ is a \textit{$G$-absolute extensor} (denoted by $G$-$\mathrm{AE}$).

A metrizable $G$-space $Y$ is a \textit{$G$-absolute neighborhood retract} (denoted by $G$-$\mathrm{ANR}$) if, for every metrizable $G$-space $X$ that contains $Y$ as a closed and invariant subset, there exists an invariant neighborhood $U$ of $Y$ in $X$ and an equivariant retraction $r:U\rightarrow Y$. If we can always take $U=X$,  we say that $Y$ is a \textit{$G$-absolute retract} (denoted by $G$-$\mathrm{AR}$).

If we take $G=\left\{e\right\}$ in the previous definitions, we obtain the classical notions of \textit{absolute neighborhood extensor} ($\mathrm{ANE}$), \textit{absolute extensor} ($\mathrm{AE}$), 
\textit{absolute neighborhood retract} ($\mathrm{ANR}$) and \textit{absolute retract} ($\mathrm{AR}$).

We will  use the following characterization, proved in \cite[Theorem 14]{sergey1987}.
\begin{theorem}[\cite{sergey1987}, Theorem 14]\label{t:G-ANE-G-ANR}
Let $X$ be a metrizable $G$-space. Then:
\begin{enumerate}[label=\upshape(\arabic*)]
    \item $X$ is a $G$-$\mathrm{ANE}$ if and only if it is a $G$-$\mathrm{ANR}$;
    \item $X$ is a $G$-$\mathrm{AE}$ if and only if it is a $G$-$\mathrm{AR}$.
\end{enumerate}
\end{theorem}

We are particularly interested in the case where the acting group is $\mathbb Z_2:=\{-1,1\}$. An \textit{involution} on a space $X$ is a continuous map $\alpha:X\rightarrow X$ satisfying $\alpha\circ\alpha=1_X$. Every action of $\mathbb Z_2$ on $X$ induces an involution $\alpha$, defined by $\alpha(x)=-1\cdot x$. Conversely, every involution on $X$ induces an action of $\mathbb Z_2$. We will denote the corresponding $\mathbb Z_2$-space by $(X,\alpha)$.

\begin{definition}\label{def:n-mean}
An \textit{$n$-mean} on a topological space $X$ is a continuous map
$p:X^{n}\rightarrow X$
satisfying the following conditions:
\begin{enumerate}
\item[(M1)] $p(x,\ldots,x)=x$ for every $x\in X$;
\item[(M2)] $p(x_{1},\ldots,x_{n})
=p(x_{\sigma(1)},\ldots,x_{\sigma(n)})$
for every $(x_{1},\ldots,x_{n})\in X^{n}$ and every permutation
$\sigma$ of $\{1,\ldots,n\}$.
\end{enumerate}

When $n=2$, we simply call $p$ a \textit{mean}.

If $X$ is a $G$-space and, in addition,
\[
p(gx_{1},\ldots,gx_{n})=g\,p(x_{1},\ldots,x_{n})
\]
for every $(x_{1},\ldots,x_{n})\in X^{n}$ and every $g\in G$, then $p$
is called a \textit{$G$-$n$-mean}, or simply an \textit{equivariant
$n$-mean} whenever the acting group is clear from the context.

When $n=2$, we simply speak of an \textit{equivariant mean}, or of a
\textit{$G$-mean} whenever it is convenient to specify the acting group.
\end{definition}

The following lemma was proved in \cite[Lemma 3.1]{JonardLopezPoo}.

\begin{lemma} \label{l: |G| divides n}
 If $p:X^{n}\rightarrow X$ is an $n$-mean and $k\in \mathbb{N}$ divides $n$,   then there exists a $k$-mean $q:X^{k}\rightarrow X$. If in addition $p$ is an equivariant $n$-mean, then $q$ can be chosen to be equivariant too. 
\end{lemma}

Combining \cite[Theorems 3.5 and 3.6]{JonardLopezPoo}, we obtain the following theorem.

\begin{theorem}\label{t: combinacion de teoremas k-mean implica G-AR}
Let $X$ be a metrizable $G$-space, where $G$ is a finite group.
Assume that $X$ admits an equivariant $m$-mean for some multiple $m$ of $|G|$. If either of the following conditions holds, then $X$ is a $G$-$\mathrm{AR}$.
\begin{enumerate}[label=\upshape(\arabic*)]
    \item $X^{H}$ is an $\mathrm{AR}$ for every subgroup $H\leq G$.
    \item $X$ is a connected $G$-$\mathrm{ANR}$ whose homology groups are finitely generated and vanish in all but finitely many dimensions.
\end{enumerate}
\end{theorem}

Let $\left(X,\leq\right)$ be a partially ordered set. If $x,y\in X$, we write $x<y$ when $x\leq y$ and $x\neq y$. If $x\leq y$, we denote the set $\left\{a\in X \mid x\leq a\leq y\right\}$ by $\left[x;y\right]$.

We say that a partially ordered set $\left(X,\leq\right)$ is a \textit{lower semilattice} if for every $x,y\in X$ there exists a greatest lower bound (denoted by $x\wedge y$) of the set $\{x,y\}$. Similarly, a partially ordered set $\left(X,\leq\right)$ is an \textit{upper semilattice} if for every $x,y\in X$ there exists a least upper bound (denoted by $x\vee y$) of the set $\{x,y\}$.

Notice that every lower semilattice $\left(X,\leq\right)$ becomes an upper semilattice when equipped with the partial order $\leq^{\prime}$, defined by
\[
x\leq^{\prime}y \quad\text{if and only if}\quad y\leq x.
\]
Conversely, every upper semilattice becomes a lower semilattice in the same way. Therefore, we will sometimes simply refer to $\left(X,\leq\right)$ as a \textit{semilattice} when it is either a lower semilattice or an upper semilattice.

A partially ordered set $\left(X,\leq\right)$ is called a \textit{lattice} if it is both a lower semilattice and an upper semilattice. We denote this structure by $\left(X,\leq,\wedge,\vee\right)$.

A lattice $\left(X,\leq,\wedge,\vee\right)$ is  said to be
\begin{itemize}
    \item \textit{complete} if every subset of $X$ has a greatest lower bound (infimum) and a least upper bound (supremum). 
    \item $\sigma$-\textit{monotonically conditionally complete} if every monotonically increasing sequence that is bounded from above has a supremum in \(X\), and every monotonically decreasing sequence that is bounded from below has an infimum in \(X\).
\end{itemize}

Let $\left(X, \leq, \wedge, \vee\right)$ be a lattice. We say that $\left(X, \leq, \wedge, \vee\right)$ is \textit{modular} if it satisfies that $$x\vee\left(a\wedge b\right)=\left(x\vee a\right)\wedge b$$ for any $x,a,b\in X$ such that $x\leq b$.

We say that $\left(X, \leq, \wedge, \vee\right)$ is \textit{distributive} if it satisfies that $$x\wedge \left(y\vee z\right)=\left(x\wedge y\right)\vee \left(x\wedge z\right)$$ for any $x,y,z\in X$. This equality is equivalent to  $$x\vee\left(y\wedge z\right)=\left(x\vee y\right)\wedge \left(x\vee z\right)$$ for any $x,y,z\in X$.

It is well known that any distributive lattice is modular (see, e.g.,
\cite[Chapter~I, Section~7]{Birkhoff}). 

Let $X$ be a Hausdorff topological space with a lattice structure $\left(X, \leq, \wedge, \vee\right)$. We say that $\left(X, \leq, \wedge, \vee\right)$ is a \textit{topological lattice} if the functions $\wedge:X\times X\rightarrow X$ and $\vee:X\times X\rightarrow X$ are continuous. Similarly, if $X$ is a topological Hausdorff space with a semilattice structure $\left(X, \leq, \wedge\right)$ (or $\left(X, \leq, \vee\right)$), we say it is a \textit{topological semilattice} if $\wedge:X\times X\rightarrow X$ is continuous (or $\vee:X\times X\rightarrow X$ is continuous).

Following \cite{BorweinThera}, a topological lattice
$(X,\leq,\wedge,\vee)$ is called a \textit{Dini lattice} if every
monotonically increasing net in $X$ that has a supremum converges to its
supremum. If the same condition is required only for monotonically
increasing sequences, we say that $(X,\leq,\wedge,\vee)$ is
\textit{countably Dini}.

Let $(X,\leq,\wedge,\vee)$ be a topological lattice. A subset $A$ of
$X$ is called \textit{convex with respect to $\leq$} if
$[x;y]\subseteq A$ whenever $x,y\in A$ and $x\leq y$. We say that $X$
is \textit{locally convex with respect to $\leq$} if its topology has a
basis consisting of convex sets.

The next lemma gives two situations in which a topological lattice is
both Dini and locally convex.

\begin{lemma}\label{l:locally convex}
Let $(X,\leq,\wedge,\vee)$ be a topological lattice. If $X$ is either
compact or locally compact and connected, then $X$ is a Dini lattice
and is locally convex with respect to $\leq$.
\end{lemma}

\begin{proof}
The Dini property follows from \cite[Theorem 10]{Lawson} and the
definition of the convex-order topology given in
\cite[p.~594, (3)]{Lawson}. 

On the other hand, local convexity follows from
\cite[Proposition 3 and Theorem 5]{Nachbin} in the compact case, and
from \cite[Theorem 1]{anderson} when $X$ is locally compact and
connected.
\end{proof}

In particular, every topological lattice that is either compact or
locally compact and connected is countably Dini.

\section{The $G$-AR property in spaces with equivariant means}

In this section, we strengthen the results stated in
Theorem~\ref{t: combinacion de teoremas k-mean implica G-AR} concerning
the $G$-$\mathrm{AR}$ property of spaces admitting equivariant means. We begin with
a useful extension result.

 \begin{lemma}\label{l: extension equivariante}
Let $G:=\{g_1,\dots,g_n\}$ be a finite group of order $n$, and let $X$ and $Z$ be $G$-spaces. Suppose that $A\subset Z$ is an invariant subset and that $f:A\to X$ is a $G$-map admitting a continuous extension $\widetilde f:Z\to X$. If $p:X^n\to X$ is an equivariant $n$-mean, then the map $F:Z\to X$ defined by
\[
F(z):=
p\bigl(
g_1^{-1}\widetilde f(g_1z),
\dots,
g_n^{-1}\widetilde f(g_nz)
\bigr),
\qquad z\in Z,
\]
is a  $G$-map extending $f$.
\end{lemma}

\begin{proof}
The continuity of $F$ follows immediately from the continuity of the action and the continuity of $\widetilde f$ and $p$.

On the other hand, if $a\in A$, since $f$ is equivariant and $\widetilde f$ is an extension of $f$, we have $\widetilde f(g_i a)=g_i f(a)$ for every $i\in\{1,\dots,n\}$. Thus, by property (M1),
\begin{align*}
F(a)
&=p\bigl(
g_1^{-1}\widetilde f(g_1a),
\dots,
g_n^{-1}\widetilde f(g_na)
\bigr)\\
&=p\bigl(
g_1^{-1}g_1f(a),
\dots,
g_n^{-1}g_nf(a)
\bigr)\\
&=p\bigl(f(a),\dots,f(a)\bigr)
=f(a),
\end{align*}
which proves that $F$ is an extension of $f$.

Finally, to prove that $F$ is $G$-equivariant, observe that for every $g\in G$, the rule $g_i\longmapsto g_i g$ defines a bijection from $G$ onto itself.  Since $g_1g,\dots,g_ng$ is a permutation of $g_1,\dots,g_n$, property (M2) and the equivariance of $p$ imply that
\begin{align*}
F(gz)
&=p\bigl(
g_1^{-1}\widetilde f(g_1gz),
\dots,
g_n^{-1}\widetilde f(g_ngz)
\bigr)\\
&=gg^{-1}p\bigl(
g_1^{-1}\widetilde f(g_1gz),
\dots,
g_n^{-1}\widetilde f(g_ngz)
\bigr)\\
&=gp\bigl(
g^{-1}g_1^{-1}\widetilde f(g_1gz),
\dots,
g^{-1}g_n^{-1}\widetilde f(g_ngz)
\bigr)\\
&=gp\bigl(
(g_1g)^{-1}\widetilde f(g_1gz),
\dots,
(g_ng)^{-1}\widetilde f(g_ngz)
\bigr)\\
&=gp\bigl(
g_1^{-1}\widetilde f(g_1z),
\dots,
g_n^{-1}\widetilde f(g_nz)
\bigr)\\
&=gF(z).
\end{align*}
Therefore, $F$ is equivariant, and the proof is complete.
\end{proof}

\begin{theorem}\label{t:ANE implies G-ANE}
Let $G$ be a finite group and let $X$ be a $G$-space. Assume that $X$ admits an equivariant $n$-mean
$p:X^{n}\rightarrow X$
for some multiple $n$ of $|G|$.
\begin{enumerate}[label=\upshape(\arabic*)]
    \item If $X$ is an $\mathrm{ANE}$, then $X$ is a $G$-$\mathrm{ANE}$.
    \item If $X$ is an $\mathrm{AE}$, then $X$ is a $G$-$\mathrm{AE}$.
\end{enumerate}
\end{theorem}

\begin{proof}
By Lemma~\ref{l: |G| divides n}, we may assume that $n=|G|$.
Let $Z$ be a metrizable $G$-space, $A\subset Z$ a closed invariant subset, and $f:A\to X$ a $G$-map.

\medskip

\noindent
(1) Suppose that $X$ is an $\mathrm{ANE}$. Then there exist an open neighborhood $V\subset Z$ of $A$ and a continuous extension
$\widetilde f:V\to X$ of $f$. By \cite[Proposition 1.1.14]{Palais2}, there exists an invariant open neighborhood $U$ of $A$ such that $U\subset V$.
Thus, $\widetilde f|_U:U\to X$ is a continuous extension of $f$. Applying Lemma~\ref{l: extension equivariante} to $\widetilde f|_U$, we obtain a $G$-map
$F:U\to X$ extending $f$. Hence, $X$ is a $G$-$\mathrm{ANE}$.

\medskip

\noindent
(2) Suppose that $X$ is an $\mathrm{AE}$. Then $f$ admits a continuous extension
$\widetilde f:Z\to X$. Applying Lemma~\ref{l: extension equivariante} directly to $\widetilde f$, we obtain a $G$-map
$F:Z\to X$ extending $f$. Therefore, $X$ is a $G$-$\mathrm{AE}$.
\end{proof}

Combining the previous theorem with Theorems~\ref{t:G-ANE-G-ANR} and
\ref{t: combinacion de teoremas k-mean implica G-AR}, we obtain the following corollary.

\begin{corollary}\label{c:mejora ANR implica G-AR}
Let $G$ be a finite group and let $X$ be a metrizable $G$-space.
Assume that $X$ admits an equivariant $n$-mean
$p:X^{n}\rightarrow X$
for some multiple $n$ of $|G|$.
\begin{enumerate}[label=\upshape(\arabic*)]
    \item If $X$ is an $\mathrm{ANR}$, then $X$ is a $G$-$\mathrm{ANR}$.
    \item If $X$ is a connected $\mathrm{ANR}$ whose homology groups are finitely generated and vanish in all but finitely many dimensions, then $X$ is a $G$-$\mathrm{AR}$. In particular, this holds if $X$ is a compact connected $\mathrm{ANR}$.
\end{enumerate}
\end{corollary}

We now strengthen Theorem~\ref{t: combinacion de teoremas k-mean implica G-AR}-(1) by showing that, when $X$ is an $\mathrm{AR}$, the existence of a suitable equivariant mean guarantees that $X^H$ is an $\mathrm{AR}$ for every subgroup $H\leq G$.

\begin{theorem}\label{t:AR+mean implica G-AR}
Let $G$ be a finite group and let $X$ be a $G$-space. Assume that $X$ admits an equivariant $n$-mean for some multiple $n$ of $|G|$.
\begin{enumerate}[label=\upshape(\arabic*)]
    \item For every subgroup $H\leq G$, the set $X^H$ is a retract of $X$.
    \item If $X$ is an $\mathrm{AR}$, then $X^H$ is an $\mathrm{AR}$ for every subgroup $H\leq G$, and $X$ is a $G$-$\mathrm{AR}$.
\end{enumerate}
\end{theorem}

\begin{proof}
To prove (1), let $H\leq G$ be a subgroup and write
$H=\{h_1,\dots,h_m\}$. By Lagrange's theorem, $m$ divides $|G|$, and therefore it also divides $n$. Hence, by Lemma~\ref{l: |G| divides n}, there exists an equivariant $m$-mean
$q:X^m\to X$.

Define $r:X\to X^H$ by
\[
r(x):=q(h_1x,\dots,h_mx).
\]
Clearly, $r$ is continuous. To verify that $r(x)\in X^H$, let $x\in X$ and $h\in H$. Notice that the rule $h_i\longmapsto hh_i$ defines a permutation of the elements of $H$. Thus, using the equivariance of $q$ together with property (M2), we obtain
\begin{align*}
hr(x)
&=hq(h_1x,\dots,h_mx)\\
&=q(hh_1x,\dots,hh_mx)\\
&=q(h_1x,\dots,h_mx)\\
&=r(x).
\end{align*}
Hence, $r(x)\in X^H$, and therefore $r$ is well defined.

It remains to prove that $r$ is a retraction. Indeed, if $x\in X^H$, then $h_ix=x$ for every $h_i\in H$. Hence, property (M1) yields
\[
r(x)=q(h_1x,\dots,h_mx)=q(x,\dots,x)=x.
\]
Thus, $X^H$ is a retract of $X$.

For (2), suppose that $X$ is an $\mathrm{AR}$. By (1), $X^H$ is a retract of $X$ for every subgroup $H\leq G$. Since every retract of an $\mathrm{AR}$ is an $\mathrm{AR}$, we conclude that $X^H$ is an $\mathrm{AR}$ for every $H\leq G$. The conclusion now follows immediately from Theorem~\ref{t: combinacion de teoremas k-mean implica G-AR}-(1).
\end{proof}

The following special cases of Theorems~\ref{t:ANE implies G-ANE} and \ref{t:AR+mean implica G-AR} will be used later.

\begin{corollary}\label{c:equivariant mean implies Z2AR}
Let $(X,\alpha)$ be a $\mathbb Z_2$-space and assume that $X$ admits
an equivariant mean. Then:
\begin{enumerate}[label=\upshape(\arabic*)]
    \item If $X$ is an $\mathrm{AE}$, then $(X,\alpha)$ is a
    $\mathbb Z_2$-$\mathrm{AE}$.
    \item If $X$ is an $\mathrm{AR}$, then $(X,\alpha)$ is a
    $\mathbb Z_2$-$\mathrm{AR}$.
\end{enumerate}
\end{corollary}

\section{Equivariant means and order}

We now turn our attention to $\mathbb{Z}_2$-spaces endowed with an order
structure. Our main goal is to use the order to construct equivariant
means and, consequently, to obtain sufficient conditions for such spaces
to be $\mathbb{Z}_2$-$\mathrm{AR}$s.

In particular, this approach is motivated by the following weak version of Jaworowski's problem.

\begin{problem}[Weak version of Jaworowski's problem]
\label{version debil del problema de Jaworowski resumido}
Let $(X,\alpha)$ be a metrizable $\mathbb{Z}_2$-space. Assume that
$X$ and $X^{\mathbb{Z}_2}$ are $\mathrm{AR}$s, and that $X$ admits a topological
lattice structure $(X,\leq,\wedge,\vee)$ such that $\alpha$ is decreasing
with respect to $\leq$. Is $(X,\alpha)$ a $\mathbb{Z}_2$-$\mathrm{AR}$?
\end{problem}

Our strategy is to identify conditions on the order structure that
guarantee the existence of an equivariant mean. In view of
Corollary~\ref{c:equivariant mean implies Z2AR}, this provides a natural
way to approach the problem above.

We begin with some terminology and basic properties.

\begin{definition}
Let $(X,\alpha)$ be a $\mathbb{Z}_2$-space and let $\leq$ be a partial order on $X$. We say that $\alpha$ is a \textit{decreasing involution} with respect to $\leq$ if
\[
x\leq y \quad\Longrightarrow\quad \alpha(y)\leq\alpha(x)
\]
for every $x,y\in X$.

If $\alpha$ is a decreasing involution with respect to $\leq$, we say that
$(X,\alpha,\leq)$ is a \textit{$\mathbb{Z}_2$-decreasing poset}.

If, in addition, $(X,\leq,\wedge)$ (respectively, $(X,\leq,\vee)$) is a topological semilattice, then we say that
$(X,\alpha,\leq,\wedge)$ (respectively, $(X,\alpha,\leq,\vee)$) is a
\textit{$\mathbb{Z}_2$-decreasing semilattice}.

When, moreover, $(X,\leq,\wedge,\vee)$ is a topological lattice, we call
$(X,\alpha,\leq,\wedge,\vee)$ a \textit{$\mathbb{Z}_2$-decreasing lattice}.
\end{definition}

Notice that what we call a decreasing involution is also referred to as an
\textit{order-reversing involution}
(see, e.g., \cite[Definition 2.2]{TerzilerSertogluPolat}).
Since an involution is bijective and $\alpha^{-1}=\alpha$, it is also an
involutory dual automorphism of the poset in the terminology of Birkhoff
\cite{Birkhoff}.
In the terminology of Artstein-Avidan and Milman, an order-reversing
involution on an ordered class of functions is called a
\textit{duality transform} (\cite[Definition 11]{ArtsteinAvidanMilman}).

Thus, in the terminology introduced above, Problem~\ref{version debil del problema de Jaworowski resumido}
asks whether every $\mathbb{Z}_2$-decreasing lattice whose underlying space
and fixed point set are $\mathrm{AR}$s must be a $\mathbb{Z}_2$-$\mathrm{AR}$.

The following De Morgan laws for order-reversing involutions are standard
(see, e.g., \cite[Proposition 3.1]{TerzilerSertogluPolat}; see also
\cite{Birkhoff}).

\begin{lemma}\label{l:decreasinglattices}
Let $(X,\alpha,\leq,\wedge,\vee)$ be a $\mathbb{Z}_2$-decreasing lattice.
Then, for every $x,y\in X$,
\[
\alpha(x\vee y)=\alpha(x)\wedge\alpha(y)
\qquad\text{and}\qquad
\alpha(x\wedge y)=\alpha(x)\vee\alpha(y).
\]
\end{lemma}

The notions of $\mathbb{Z}_2$-decreasing semilattice and
$\mathbb{Z}_2$-decreasing lattice are closely related. In fact, in the
presence of a decreasing involution, either semilattice operation
determines the other, as shown in the following lemma.

\begin{lemma}\label{l:semilattice implies lattice}
Let $(X,\alpha)$ be a $\mathbb{Z}_2$-space. If
$(X,\alpha,\leq,\wedge)$ is a $\mathbb{Z}_2$-decreasing semilattice,
then there exists a continuous operation
$\vee:X\times X\to X$
such that $(X,\alpha,\leq,\wedge,\vee)$ is a
$\mathbb{Z}_2$-decreasing lattice.

Similarly, if $(X,\alpha,\leq,\vee)$ is a
$\mathbb{Z}_2$-decreasing semilattice, then there exists a continuous
operation $\wedge:X\times X\to X$ such that
$(X,\alpha,\leq,\wedge,\vee)$ is a
$\mathbb{Z}_2$-decreasing lattice.
\end{lemma}

\begin{proof}
Assume that $(X,\alpha,\leq,\wedge)$ is a $\mathbb{Z}_2$-decreasing semilattice. Define
$\vee:X\times X\to X$ by
\[
x\vee y:=\alpha\bigl(\alpha(x)\wedge\alpha(y)\bigr).
\]
Since $\wedge$ and $\alpha$ are continuous, $\vee$ is also continuous.

Let $x,y\in X$. We will show that $x\vee y$ is the least upper bound of $\{x,y\}$. Since
\[
\alpha(x)\wedge\alpha(y)\leq\alpha(x)
\qquad\text{and}\qquad
\alpha(x)\wedge\alpha(y)\leq\alpha(y),
\]
and $\alpha$ is decreasing with respect to $\leq$, we obtain
\[
x\leq \alpha\bigl(\alpha(x)\wedge\alpha(y)\bigr)
\qquad\text{and}\qquad
y\leq \alpha\bigl(\alpha(x)\wedge\alpha(y)\bigr).
\]
Thus, $\alpha\bigl(\alpha(x)\wedge\alpha(y)\bigr)$ is an upper bound of $\{x,y\}$.

Now, suppose that $s\in X$ is another upper bound of $\{x,y\}$. Since
$x\leq s$ and $y\leq s$, and $\alpha$ is decreasing, we have
\[
\alpha(s)\leq\alpha(x)
\qquad\text{and}\qquad
\alpha(s)\leq\alpha(y).
\]
Therefore, $\alpha(s)\leq\alpha(x)\wedge\alpha(y)$, and applying $\alpha$ once more gives
\[
\alpha\bigl(\alpha(x)\wedge\alpha(y)\bigr)\leq s.
\]
Hence, $\alpha\bigl(\alpha(x)\wedge\alpha(y)\bigr)$ is the least upper bound of $\{x,y\}$, as desired.

Thus, $(X,\leq,\wedge,\vee)$ is a lattice and, since both $\wedge$ and $\vee$ are continuous, it is a topological lattice. Therefore, $(X,\alpha,\leq,\wedge,\vee)$ is a $\mathbb{Z}_2$-decreasing lattice.

The case in which $(X,\alpha,\leq,\vee)$ is a $\mathbb{Z}_2$-decreasing semilattice is analogous.
\end{proof}

By Lemma~\ref{l:semilattice implies lattice}, every result proved for
$\mathbb{Z}_2$-decreasing lattices also applies to
$\mathbb{Z}_2$-decreasing semilattices.

\subsection{$\mathbb{Z}_2$-decreasing modular lattices.}

We begin with the case of modular lattices. In this setting, the modular
identity allows us to construct an equivariant mean whenever the
involution has a fixed point. As a consequence, we obtain a positive
answer to the weak version of Jaworowski's problem for
$\mathbb{Z}_2$-decreasing modular lattices.

\begin{theorem}
Let $(X,\alpha,\leq,\wedge,\vee)$ be a $\mathbb{Z}_2$-decreasing lattice.
\begin{enumerate}[label=\upshape(\arabic*)]
\item If there exists $\theta\in X^{\mathbb{Z}_2}$ such that, for every $a,b\in X$ with $a\leq b$,
\begin{equation}\label{eqmodular}
a\vee(\theta\wedge b)=(a\vee\theta)\wedge b,
\end{equation}
then the map $p:X\times X\to X$ defined by
\[
p(x,y):=(x\wedge y)\vee\bigl(\theta\wedge(x\vee y)\bigr)
\]
is an equivariant mean on $X$.

\item If, in addition, $X$ is an $\mathrm{AR}$ ($\mathrm{AE}$), then $(X,\alpha)$ is a $\mathbb{Z}_2$-$\mathrm{AR}$ ($\mathbb{Z}_2$-$\mathrm{AE}$).
\end{enumerate}
\end{theorem}

\begin{proof}
(1) Notice that $p$ is symmetric and satisfies $p(x,x)=x$ for every $x\in X$. Also, since $(X,\leq,\wedge,\vee)$ is a topological lattice, $p$ is continuous. Hence, $p$ is a mean.

Let $x,y\in X$. Applying~\eqref{eqmodular} with $a=x\wedge y$ and
$b=x\vee y$, and using the fact that $\theta$ is fixed by $\alpha$, we obtain
\begin{align*}
p(x,y)
&=(x\wedge y)\vee\bigl(\theta\wedge(x\vee y)\bigr)\\
&=\bigl((x\wedge y)\vee\theta\bigr)\wedge(x\vee y)\\
&=\alpha\circ\alpha\bigl(
\bigl((x\wedge y)\vee\theta\bigr)\wedge(x\vee y)
\bigr)\\
&=\alpha\bigl(
\alpha\bigl((x\wedge y)\vee\theta\bigr)
\vee\alpha(x\vee y)
\bigr)\\
&=\alpha\bigl(
\bigl(\alpha(x\wedge y)\wedge\alpha(\theta)\bigr)
\vee
\bigl(\alpha(x)\wedge\alpha(y)\bigr)
\bigr)\\
&=\alpha\bigl(
\bigl((\alpha(x)\vee\alpha(y))\wedge\theta\bigr)
\vee
\bigl(\alpha(x)\wedge\alpha(y)\bigr)
\bigr)\\
&=\alpha\bigl(p(\alpha(x),\alpha(y))\bigr).
\end{align*}
Therefore, $p$ is an equivariant mean.

(2) Suppose that $X$ is an $\mathrm{AR}$ ($\mathrm{AE}$). By (1), $X$ admits an equivariant mean. Therefore, Corollary~\ref{c:equivariant mean implies Z2AR}  implies that $(X,\alpha)$ is a $\mathbb{Z}_2$-$\mathrm{AR}$ ($\mathbb{Z}_2$-$\mathrm{AE}$).
\end{proof}

Since every point of a modular lattice satisfies condition~\eqref{eqmodular}, and every distributive lattice is modular, we obtain the following corollary.

\begin{corollary}\label{c:caso modular}
Let $(X,\alpha,\leq,\wedge,\vee)$ be a
$\mathbb{Z}_2$-decreasing lattice such that $X$ is an
$\mathrm{AR}$ (respectively, an $\mathrm{AE}$). If
$(X,\leq,\wedge,\vee)$ is modular (in particular, if it is
distributive) and $X^{\mathbb{Z}_2}\neq\emptyset$, then $(X,\alpha)$
is a $\mathbb{Z}_2$-$\mathrm{AR}$ (respectively, a
$\mathbb{Z}_2$-$\mathrm{AE}$). Moreover, in the $\mathrm{AR}$ case,
the assumption $X^{\mathbb{Z}_2}\neq\emptyset$ may be omitted when
$X$ is compact.
\end{corollary}

\begin{proof}
The first assertion follows from the preceding theorem. For the final
assertion, suppose that $X$ is a compact $\mathrm{AR}$. By
\cite[Corollary~3.5.4]{vanmill}, $X$ has the fixed point property.
Therefore, $\alpha$ has a fixed point, and hence
$X^{\mathbb{Z}_2}\neq\emptyset$.
\end{proof}

\begin{example}
Consider the Hilbert cube $Q$ and the topological lattice structure
$(Q,\leq,\wedge,\vee)$, where $\leq$ is the partial order defined by
\[
(x_n)\leq(y_n)
\quad\text{if and only if}\quad
x_n\leq y_n \text{ for every } n\in\mathbb{N},
\]
and the operations $\wedge$ and $\vee$ are given by
\[
(x_n)\wedge(y_n)
=
\bigl(\min\{x_n,y_n\}\bigr),
\qquad
(x_n)\vee(y_n)
=
\bigl(\max\{x_n,y_n\}\bigr).
\]
Notice that $(Q,\leq,\wedge,\vee)$ is a distributive topological lattice. Therefore, by Corollary~\ref{c:caso modular}, $(Q,\alpha)$ is a $\mathbb{Z}_2$-$\mathrm{AR}$ for every involution $\alpha:Q\to Q$ that is decreasing with respect to $\leq$.
$\blacktriangleleft$
\end{example}

\subsection{Generalized geometric means}

Inspired by the Babylonian method for computing the geometric mean and
by the work of V.~Milman and L.~Rotem \cite{MilmanRotem}, we introduce
the following notions.

\begin{definition}\label{strict order preserving}
Let $X$ be a topological space endowed with a partial order $\leq$, and
let $p:X\times X\rightarrow X$ be a mean. We say that $p$ is an
\textit{order-preserving mean} if
\[
x\leq p(x,y)\leq y
\]
whenever $x\leq y$.

An order-preserving mean $p$ is called
\begin{itemize}
    \item \textit{left-strict} if
    \[
    x<p(x,y)\leq y
    \]
    whenever $x<y$;

    \item \textit{right-strict} if
    \[
    x\leq p(x,y)<y
    \]
    whenever $x<y$.
\end{itemize}

We abbreviate these conditions by saying that $p$ is an
\textit{$\mathrm{LS}$-order-preserving mean} or an
\textit{$\mathrm{RS}$-order-preserving mean}, respectively. If $p$ is both
left-strict and right-strict, we simply say that it is a
\textit{strict order-preserving mean}.
\end{definition}

\begin{definition}\label{def:arithmetic-harmonic-Babylonian}
Let $(X,\alpha,\leq)$ be a $\mathbb Z_2$-decreasing poset, and let
$p:X\times X\rightarrow X$ be a mean. Define the mean
$p^{\alpha}:X\times X\rightarrow X$ by
\[
p^{\alpha}(x,y)
:=
\alpha\bigl(p(\alpha(x),\alpha(y))\bigr).
\]

We introduce the following terminology:
\begin{enumerate}
    \item We say that $p$ is an \textit{arithmetic mean} with respect to
    $\alpha$ if
    \[
    p^{\alpha}(x,y)\leq p(x,y)
    \]
    for every $x,y\in X$.

    \item We say that $p$ is a \textit{harmonic mean} with respect to
    $\alpha$ if
    \[
    p(x,y)\leq p^{\alpha}(x,y)
    \]
    for every $x,y\in X$.

    \item For every $n\geq 0$, define maps
$a_n,h_n:X\times X\rightarrow X$ recursively by
\[
a_0(x,y)=x,\qquad h_0(x,y)=y,
\]
and
\[
a_{n+1}(x,y)
=p\bigl(a_n(x,y),h_n(x,y)\bigr),
\]
\[
h_{n+1}(x,y)
=p^\alpha\bigl(a_n(x,y),h_n(x,y)\bigr).
\]

We say that $p$ is \textit{Babylonian} if one of the sequences of maps
$(a_n)_{n\geq 1}$ and $(h_n)_{n\geq 1}$ is monotonically increasing
with respect to the pointwise order and the other is monotonically
decreasing, and if, for every $x,y\in X$, the sequences
$\bigl(a_n(x,y)\bigr)$ and $\bigl(h_n(x,y)\bigr)$ converge and satisfy
\[
\lim_{n\rightarrow\infty}a_n(x,y)
=
\lim_{n\rightarrow\infty}h_n(x,y).
\]
\end{enumerate}
\end{definition}

\begin{lemma}\label{l:properties-dual-mean}
Let $(X,\alpha,\leq)$ be a $\mathbb Z_2$-decreasing poset, and let
$p:X\times X\rightarrow X$ be a mean. Then:
\begin{enumerate}[label=\upshape(\arabic*)]
    \item $\left(p^\alpha\right)^\alpha=p$.
    \item $p$ is an arithmetic mean with respect to $\alpha$ if and only
    if $p^\alpha$ is a harmonic mean with respect to $\alpha$.
    \item $p$ is an $\mathrm{RS}$-order-preserving mean if and only if $p^\alpha$
    is an $\mathrm{LS}$-order-preserving mean.
    \item $p$ is Babylonian if and only if $p^\alpha$ is Babylonian.
\end{enumerate}
\end{lemma}

\begin{proof}
For every $x,y\in X$, we have
\begin{align*}
\left(p^\alpha\right)^\alpha(x,y)
&=\alpha\bigl(
p^\alpha(\alpha(x),\alpha(y))
\bigr)\\
&=\alpha\Bigl(
\alpha\bigl(
p(\alpha(\alpha(x)),\alpha(\alpha(y)))
\bigr)
\Bigr)\\
&=p(x,y).
\end{align*}
This proves (1).

By (1),
\[
p^\alpha(x,y)\leq p(x,y)
\quad\Longleftrightarrow\quad
p^\alpha(x,y)\leq
\left(p^\alpha\right)^\alpha(x,y)
\]
for every $x,y\in X$. Thus, $p$ is arithmetic if and only if
$p^\alpha$ is harmonic, proving (2).
To prove (3), suppose first that $p$ is an $\mathrm{RS}$-order-preserving mean. Let
$x,y\in X$ with $x<y$. Since $\alpha$ is a decreasing involution,
\[
\alpha(y)<\alpha(x).
\]
Therefore, by the symmetry and the $\mathrm{RS}$-order-preserving property of $p$,
we obtain
\[
\alpha(y)
\leq p\bigl(\alpha(y),\alpha(x)\bigr)
=p\bigl(\alpha(x),\alpha(y)\bigr)
<\alpha(x).
\]
Applying $\alpha$ gives
\[
x<
\alpha\bigl(p(\alpha(x),\alpha(y))\bigr)
\leq y.
\]
Thus, $x<p^\alpha(x,y)\leq y$, and therefore $p^\alpha$ is an $\mathrm{LS}$-order-preserving mean.

Conversely, suppose that $p^\alpha$ is an $\mathrm{LS}$-order-preserving mean. If
$x<y$, then $\alpha(y)<\alpha(x)$, and hence
\[
\alpha(y)
<
p^\alpha\bigl(\alpha(y),\alpha(x)\bigr)
\leq\alpha(x).
\]
Since $p^\alpha$ is symmetric and $p^\alpha\bigl(\alpha(x),\alpha(y)\bigr)
=\alpha\bigl(p(x,y)\bigr)$, applying $\alpha$ yields $x\leq p(x,y)<y$. Therefore, $p$ is an $\mathrm{RS}$-order-preserving mean.

(4) Let $(a_n)_{n\geq 1}$ and $(h_n)_{n\geq 1}$ be the sequences of
maps associated with $p$, and let $(\widetilde a_n)_{n\geq 1}$ and
$(\widetilde h_n)_{n\geq 1}$ be those associated with $p^\alpha$.
Using (1), the symmetry of $p$ and $p^\alpha$, and the recursive
definitions of these maps, an induction gives
\[
\widetilde a_n=h_n
\qquad\text{and}\qquad
\widetilde h_n=a_n
\]
for every $n\geq 1$. Thus, the two sequences associated with
$p^\alpha$ are precisely those associated with $p$ with their roles
interchanged. Consequently, $p$ is Babylonian if and only if
$p^\alpha$ is Babylonian.
\end{proof}

It follows from Lemma~\ref{l:properties-dual-mean} that every
$\mathbb Z_2$-decreasing poset admitting an arithmetic (respectively,
harmonic) mean also admits a harmonic (respectively, arithmetic) mean.
Likewise, such a poset admits an $\mathrm{RS}$-order-preserving mean if and only if
it admits an $\mathrm{LS}$-order-preserving mean. The following example shows that
every $\mathbb Z_2$-decreasing lattice admits means satisfying each of
these properties.

\begin{example}\label{ex:lattice-means}
Let $(X,\alpha,\leq,\wedge,\vee)$ be a
$\mathbb Z_2$-decreasing lattice. Then the lattice operations
\[
\vee,\wedge:X\times X\rightarrow X
\]
are, respectively, an $\mathrm{LS}$-order-preserving arithmetic mean and an
$\mathrm{RS}$-order-preserving harmonic mean. \(\blacktriangleleft\)
\end{example}

The following theorem shows that, on a locally convex
$\mathbb Z_2$-decreasing lattice, the common limit determined by a
Babylonian mean defines an equivariant mean.

\begin{theorem}\label{t:Babylonian-mean-implies-equivariant-mean}
Let $(X,\alpha,\leq,\wedge,\vee)$ be a
$\mathbb Z_2$-decreasing lattice that is locally convex with respect to
$\leq$, and let $q:X\times X\rightarrow X$ be a Babylonian mean. Then
the common limit
\[
p(x,y):=
\lim_{n\rightarrow\infty}a_n(x,y)
=
\lim_{n\rightarrow\infty}h_n(x,y),
\]
where $(a_n)$ and $(h_n)$ are the sequences of Definition~\ref{def:arithmetic-harmonic-Babylonian}-(3) associated with $q$, defines a $\mathbb Z_2$-equivariant mean
$p:X\times X\rightarrow X$.

Additionally, if $X$ is an $\mathrm{AR}$ ($\mathrm{AE}$), then $(X,\alpha)$ is a
$\mathbb Z_2$-$\mathrm{AR}$ ($\mathbb Z_2$-$\mathrm{AE}$).
\end{theorem}

\begin{proof}
Let us consider the case in which $(a_n)_{n\geq 1}$ is monotonically
decreasing and $(h_n)_{n\geq 1}$ is monotonically increasing. The other
case is analogous, with the roles of the two sequences interchanged.

We first verify that $p$ satisfies conditions (M1) and (M2) and that it
is equivariant. Since both $q$ and
$q^\alpha$ are means, an inductive argument shows that
\[
a_n(x,x)=h_n(x,x)=x
\]
for every $x\in X$ and every $n\geq 0$. Consequently, $p(x,x)=x$ for every $x\in X$.

Moreover, $a_1=q$ and $h_1=q^\alpha$ are symmetric. The recursive
definition of $a_n$ and $h_n$ shows, by induction, that both
$a_n$ and $h_n$ are symmetric for every $n\geq 1$, and therefore $p$
is symmetric.

To prove that $p$ is equivariant, we use the identities
\[
a_n\bigl(\alpha(x),\alpha(y)\bigr)
=\alpha\bigl(h_n(x,y)\bigr) \quad\text{and}\quad h_n\bigl(\alpha(x),\alpha(y)\bigr)
=\alpha\bigl(a_n(x,y)\bigr),
\]
which hold for every $n\geq 1$ and every $x,y\in X$, as can be verified
by induction on $n$. Hence, using the continuity of $\alpha$, we obtain

\begin{align*}
p\bigl(\alpha(x),\alpha(y)\bigr)
&=\lim_{n\rightarrow\infty}
  a_n\bigl(\alpha(x),\alpha(y)\bigr)\\
&=\lim_{n\rightarrow\infty}\alpha\bigl(h_n(x,y)\bigr)\\
&=\alpha\left(\lim_{n\rightarrow\infty}h_n(x,y)\right)\\
&=\alpha\bigl(p(x,y)\bigr).
\end{align*}
Thus, $p$ is $\mathbb Z_2$-equivariant.

It remains to prove that $p$ is continuous. First, notice that $a_n$ and
$h_n$ are continuous for every $n\geq 0$. Indeed, $a_0$ and $h_0$ are
the coordinate projections, and the conclusion follows inductively
from the continuity of $q$ and $q^\alpha$.

We claim that
\[
h_m(x,y)\leq p(x,y)\leq a_m(x,y)
\]
for every $m\geq 1$ and every $(x,y)\in X\times X$. Fix $m\geq 1$ and
$(x,y)\in X\times X$. Since $(h_n)_{n\geq 1}$ is monotonically
increasing,
\[
h_m(x,y)\wedge h_n(x,y)=h_m(x,y)
\]
whenever $n\geq m$. Taking the limit as $n\rightarrow\infty$ and using
the continuity of $\wedge$, we obtain
\[
h_m(x,y)\wedge p(x,y)=h_m(x,y),
\]
and hence \(h_m(x,y)\leq p(x,y)\).

An analogous argument, using $\vee$ and the fact that
$(a_n)_{n\geq 1}$ is monotonically decreasing, shows that
$p(x,y)\leq a_m(x,y)$.  Therefore,
\[
p(x,y)\in [h_m(x,y);a_m(x,y)]
\]
as desired.

Now, fix $(u,w)\in X\times X$, and let $U\subseteq X$ be an open
neighborhood of $p(u,w)$. By local convexity, we may choose an open
convex set $V$ with $p(u,w)\in V\subseteq U$.

The convergence of $\bigl(a_n(u,w)\bigr)$ and
$\bigl(h_n(u,w)\bigr)$ allows us to choose $N\geq 1$ with both
$a_N(u,w)$ and $h_N(u,w)$ in $V$. Since $a_N$ and $h_N$ are continuous,
there exists an open neighborhood $W$ of $(u,w)$ in $X\times X$ for
which $a_N(x,y),h_N(x,y)\in V$ whenever $(x,y)\in W$. The convexity of $V$ now implies that
\[
p(x,y)\in[h_N(x,y);a_N(x,y)]\subseteq V\subseteq U
\]
for every $(x,y)\in W$. Thus, $p$ is continuous.

We conclude that $p$ is an equivariant mean. Finally, if $X$ is an $\mathrm{AR}$ ($\mathrm{AE}$),
Corollary~\ref{c:equivariant mean implies Z2AR}  implies that
$(X,\alpha)$ is a $\mathbb Z_2$-$\mathrm{AR}$ ( $\mathbb Z_2$-$\mathrm{AE}$).
\end{proof}

Combining Lemma~\ref{l:locally convex} with
Theorem~\ref{t:Babylonian-mean-implies-equivariant-mean}, we immediately
obtain the following corollary.

\begin{corollary}\label{c:Babylonian-mean-compact-lattice}
Let $(X,\alpha,\leq,\wedge,\vee)$ be a
$\mathbb Z_2$-decreasing lattice whose underlying space is either
compact, or locally compact and connected. If $X$ admits a Babylonian
mean, then it admits an equivariant mean. If, in addition, $X$ is an $\mathrm{AR}$ ($\mathrm{AE}$),
then $(X,\alpha)$ is a $\mathbb Z_2$-$\mathrm{AR}$ ($\mathbb Z_2$-$\mathrm{AE}$).
\end{corollary}

Before proving the next result, we stablish the following fact.

\begin{proposition}\label{p:duality-monotone-sequences}
Let $(X,\alpha,\leq,\wedge,\vee)$ be a
$\mathbb Z_2$-decreasing lattice. Then the following properties hold.

\begin{enumerate} [label=\upshape(\arabic*)]
\item If $X$ is countably Dini, then every monotonically
decreasing sequence in $X$ that has an infimum converges to its
infimum.

\item The following conditions are equivalent:
    \begin{enumerate}
    \item[\upshape(a)] $X$ is $\sigma$-monotonically conditionally complete;
    \item[\upshape(b)] every monotonically increasing sequence in $X$ that is
    bounded from above has a supremum;
    \item[\upshape(c)] every monotonically decreasing sequence in $X$ that is
    bounded from below has an infimum.
    \end{enumerate}
\end{enumerate}
\end{proposition}

\begin{proof}
For (1), let $(x_n)_{n\geq1}$ be monotonically decreasing and suppose
that $x=\inf_{n\geq1}x_n.$ Since $\alpha$ is a decreasing involution, the sequence
$(\alpha(x_n))_{n\geq1}$ is monotonically increasing and
\[
\alpha(x)=\sup_{n\geq1}\alpha(x_n).
\]
The countably Dini property therefore implies that $\alpha(x_n)\longrightarrow\alpha(x).$ Applying the continuous map $\alpha$, we obtain $x_n\to x$.

For (2), condition (a) clearly implies both (b) and (c). Suppose now
that (b) holds, and let $(x_n)_{n\geq1}$ be a monotonically decreasing
sequence bounded from below by some $b\in X$. Then
$(\alpha(x_n))_{n\geq1}$ is monotonically increasing and bounded from
above by $\alpha(b)$. By (b), the element
\[
s=\sup_{n\geq1}\alpha(x_n)
\]
exists. We claim that $\alpha(s)=\inf_{n\geq1}x_n$. Indeed, the inequalities $\alpha(x_n)\leq s$ imply that
$\alpha(s)\leq x_n$ for every $n$, so $\alpha(s)$ is a lower bound of
$(x_n)$. If $y$ is any other lower bound, then
$y\leq x_n$ for every $n$, and hence
\[
\alpha(x_n)\leq\alpha(y).
\]
It follows that $s\leq\alpha(y)$ and therefore $y\leq\alpha(s)$.
Thus, $\alpha(s)$ is the infimum of $(x_n)$, and (c) holds.

Conversely, suppose that (c) holds, and let $(x_n)_{n\geq1}$ be a
monotonically increasing sequence bounded from above by some $b\in X$.
Then $(\alpha(x_n))_{n\geq1}$ is monotonically decreasing and bounded
from below by $\alpha(b)$. If $t=\inf_{n\geq1}\alpha(x_n)$, the same order-reversing argument shows that $\alpha(t)=\sup_{n\geq1}x_n$. Hence (b) holds, and the three conditions are equivalent.
\end{proof}

\begin{proposition}\label{p:Babylonian mean}
Let $(X,\alpha,\leq,\wedge,\vee)$ be a
$\mathbb Z_2$-decreasing lattice. Assume that $X$ is
$\sigma$-monotonically conditionally complete and countably Dini.
If $q:X\times X\rightarrow X$ is either an $\mathrm{RS}$-order-preserving
arithmetic mean or an $\mathrm{LS}$-order-preserving harmonic mean, then $q$ is
Babylonian.
\end{proposition}

\begin{proof}

It is enough to consider the case in which $q$ is an
$\mathrm{RS}$-order-preserving arithmetic mean. Indeed, if $q$ is an
$\mathrm{LS}$-order-preserving harmonic mean, then
Lemma~\ref{l:properties-dual-mean} implies that $q^\alpha$ is an
$\mathrm{RS}$-order-preserving arithmetic mean, and that $q$ is Babylonian if and
only if $q^\alpha$ is Babylonian.

Consider the sequences of maps $(a_n)_{n\geq 1}$ and
$(h_n)_{n\geq 1}$ associated with $q$, as in
Definition~\ref{def:arithmetic-harmonic-Babylonian}. We will prove that
$(a_n)_{n\geq 1}$ is monotonically decreasing and
$(h_n)_{n\geq 1}$ is monotonically increasing, that both sequences
converge pointwise, and that
\[
\lim_{n\rightarrow\infty}a_n(x,y)
=
\lim_{n\rightarrow\infty}h_n(x,y)
\]
for every $x,y\in X$.

Fix $(x,y)\in X\times X$. Since $q$ is arithmetic, the recursive
definitions of $a_n$ and $h_n$ imply that
\[
h_n(x,y)\leq a_n(x,y)
\]
for every $n\geq 1$. Since $q$ is order preserving and symmetric, we
obtain
\begin{equation}\label{eq:decreasing-an}
\begin{aligned}
h_n(x,y)
&\leq q\bigl(h_n(x,y),a_n(x,y)\bigr)\\
&=a_{n+1}(x,y)
\leq a_n(x,y).
\end{aligned}
\end{equation}

Moreover, since $\alpha$ is decreasing, $\alpha\bigl(a_n(x,y)\bigr)
\leq
\alpha\bigl(h_n(x,y)\bigr)$. The fact that $q$ is order preserving implies
\[
\alpha\bigl(a_n(x,y)\bigr)
\leq
q\bigl(\alpha(a_n(x,y)),\alpha(h_n(x,y))\bigr)
\leq
\alpha\bigl(h_n(x,y)\bigr).
\]
Applying $\alpha$ to these inequalities, we obtain
\begin{equation}\label{eq:increasing-hn}
h_n(x,y)
\leq h_{n+1}(x,y)
\leq a_n(x,y).
\end{equation}
Since $(x,y)$ was arbitrary, inequialities \eqref{eq:decreasing-an} and
\eqref{eq:increasing-hn} show that $(a_n)_{n\geq 1}$ is monotonically
decreasing and $(h_n)_{n\geq 1}$ is monotonically increasing with
respect to the pointwise order. Furthermore, for every $n\geq 1$,
\[
h_1(x,y)\leq h_n(x,y)\leq a_n(x,y)\leq a_1(x,y).
\]
Thus, the increasing sequence $(h_n(x,y))_{n\geq 1}$ is bounded above,
whereas the decreasing sequence $(a_n(x,y))_{n\geq 1}$ is bounded
below. Since $X$ is $\sigma$-monotonically conditionally complete, the
elements
\[
h_\infty(x,y):=\sup_{n\geq 1}h_n(x,y)
\qquad\text{and}\qquad
a_\infty(x,y):=\inf_{n\geq 1}a_n(x,y)
\]
exist. Then, the countably Dini property in combination with Proposition~\ref{p:duality-monotone-sequences} gives
\[
\lim_{n\rightarrow\infty}h_n(x,y)=h_\infty(x,y) \qquad\text{and}\qquad \lim_{n\rightarrow\infty}a_n(x,y)=a_\infty(x,y).
\]
Therefore, it only remains to prove that
$a_\infty(x,y)=h_\infty(x,y)$. To prove this, notice first that the continuity of $q$ gives
\begin{equation} \label{eqainf}
\begin{aligned}
a_\infty(x,y)
&=\lim_{n\rightarrow\infty}a_{n+1}(x,y)\\
&=\lim_{n\rightarrow\infty}
q\bigl(a_n(x,y),h_n(x,y)\bigr)\\
&=q\left(
\lim_{n\rightarrow\infty}a_n(x,y),
\lim_{n\rightarrow\infty}h_n(x,y)
\right)\\
&=q\bigl(a_\infty(x,y),h_\infty(x,y)\bigr).
\end{aligned}
\end{equation}
Similarly, using the continuity of $q^\alpha$, we obtain
\begin{align*}
h_\infty(x,y)
&=\lim_{n\rightarrow\infty}h_{n+1}(x,y)\\
&=\lim_{n\rightarrow\infty}
q^\alpha\bigl(a_n(x,y),h_n(x,y)\bigr)\\
&=q^\alpha\left(
\lim_{n\rightarrow\infty}a_n(x,y),
\lim_{n\rightarrow\infty}h_n(x,y)
\right)\\
&=q^\alpha\bigl(a_\infty(x,y),h_\infty(x,y)\bigr).
\end{align*}
Since $q$ is arithmetic,
\[
h_\infty(x,y)=q^\alpha\bigl(a_\infty(x,y),h_\infty(x,y)\bigr)
\leq
q\bigl(a_\infty(x,y),h_\infty(x,y)\bigr)=a_\infty(x,y).
\]

Suppose that this inequality is strict. Since $q$ is an
$\mathrm{RS}$-order-preserving mean, we have
\[
q\bigl(h_\infty(x,y),a_\infty(x,y)\bigr)
<a_\infty(x,y).
\]
On the other hand, the symmetry of $q$ and (\ref{eqainf}) give
\[
q\bigl(h_\infty(x,y),a_\infty(x,y)\bigr)
=q\bigl(a_\infty(x,y),h_\infty(x,y)\bigr)
=a_\infty(x,y),
\]
which is a contradiction. Therefore, $a_\infty(x,y)=h_\infty(x,y)$. Since $(x,y)$ was arbitrary, we conclude that
\[
\lim_{n\rightarrow\infty}a_n(x,y)
=
\lim_{n\rightarrow\infty}h_n(x,y)
\]
for every $x,y\in X$. Hence, $q$ is Babylonian.
\end{proof}

\begin{remark}
Let $(X,\alpha,\leq,\wedge,\vee)$ be a
$\mathbb Z_2$-decreasing lattice containing points $x,y\in X$ with
$x<y$.

The mean $q=\vee$ is $\mathrm{LS}$-order preserving and arithmetic, while
$q^\alpha=\wedge$. For the sequences associated with $q$, we have
\[
a_1(x,y)=x\vee y=y
\qquad\text{and}\qquad
h_1(x,y)=x\wedge y=x.
\]
It follows that
\[
a_n(x,y)=y
\qquad\text{and}\qquad
h_n(x,y)=x
\]
for every $n\geq 1$. Hence, $\vee$ is not Babylonian.

Similarly, the mean $q=\wedge$ is $\mathrm{RS}$-order preserving and harmonic,
while $q^\alpha=\vee$. In this case,
\[
a_n(x,y)=x
\qquad\text{and}\qquad
h_n(x,y)=y
\]
for every $n\geq 0$, and therefore $\wedge$ is not Babylonian.

Consequently, the $\mathrm{RS}$ condition in the arithmetic case and the $\mathrm{LS}$
condition in the harmonic case in
Proposition~\ref{p:Babylonian mean} cannot be interchanged.
\end{remark}

\begin{theorem}\label{t:order-preserving mean implies equivariant mean}
Let $(X,\alpha,\leq,\wedge,\vee)$ be a
$\mathbb Z_2$-decreasing lattice. Assume that $X$ is locally convex
with respect to $\leq$, countably Dini, and
$\sigma$-monotonically conditionally complete.

If $X$ admits either an $\mathrm{RS}$-order-preserving arithmetic mean or an
$\mathrm{LS}$-order-preserving harmonic mean, then $X$ admits an equivariant mean.

If, in addition, $X$ is an $\mathrm{AR}$ (respectively, an $\mathrm{AE}$), then
$(X,\alpha)$ is a $\mathbb Z_2$-$\mathrm{AR}$ (respectively, a
$\mathbb Z_2$-$\mathrm{AE}$).
\end{theorem}

  \begin{proof}
By Proposition~\ref{p:Babylonian mean}, either of the means in the
hypothesis is Babylonian. The existence of an equivariant mean then
follows from
Theorem~\ref{t:Babylonian-mean-implies-equivariant-mean}. The remaining
assertions follow from
Corollary~\ref{c:equivariant mean implies Z2AR}.
\end{proof}

\begin{remark}\label{r:sigma-monotonically-complete}
Let $(X,\leq,\wedge,\vee)$ be a topological lattice. Then $X$ is
$\sigma$-monotonically conditionally complete whenever any of the
following conditions holds:
\begin{enumerate}
    \item $X$ is a complete lattice;
    \item $X$ is compact;
    \item every order interval in $X$ is compact.
\end{enumerate}
\end{remark}

\begin{proof}
Assertion (1) follows directly from the definitions.

For (2), \cite[Proposition VI-1.13(v)]{GierzEtAl}, applied to both
semilattice structures of $X$, implies that $X$ is a complete lattice.
Hence, assertion (1) applies.

Finally, suppose that every order interval in $X$ is compact. Each
order interval is a topological sublattice of $X$ and is therefore
complete by (2). Every monotonically increasing sequence bounded from
above is contained in an order interval of the form $[x_1;u]$, and its
supremum in this interval is also its supremum in $X$. The dual
argument applies to monotonically decreasing sequences. Thus, $X$ is
$\sigma$-monotonically conditionally complete.
\end{proof}

Combining Lemma~\ref{l:locally convex}, Theorem~\ref{t:order-preserving mean implies equivariant mean} and Remark~\ref{r:sigma-monotonically-complete}, we obtain the following corollary.

\begin{corollary}\label{c:locally-compact-lattice-equivariant-mean}
Let $(X,\alpha,\leq,\wedge,\vee)$ be a
$\mathbb Z_2$-decreasing lattice. Assume that $X$ is locally compact
and connected, and that every order interval $[x;y]$ in $X$ is
compact. If $X$ admits either an $\mathrm{RS}$-order-preserving arithmetic mean
or an $\mathrm{LS}$-order-preserving harmonic mean, then $X$ admits an
equivariant mean.

If, in addition, $X$ is an $\mathrm{AR}$ (respectively, an $\mathrm{AE}$), then
$(X,\alpha)$ is a $\mathbb Z_2$-$\mathrm{AR}$ (respectively, a
$\mathbb Z_2$-$\mathrm{AE}$).
\end{corollary}

\begin{corollary}\label{c:compact-lattice-equivariant-mean}
Let $(X,\alpha,\leq,\wedge,\vee)$ be a compact
$\mathbb Z_2$-decreasing lattice. If $X$ admits either an
$\mathrm{RS}$-order-preserving arithmetic mean or an $\mathrm{LS}$-order-preserving harmonic
mean, then $X$ admits an equivariant mean. Moreover:
\begin{enumerate}[label=\upshape(\arabic*)]
    \item if $X$ is an $\mathrm{AE}$, then $(X,\alpha)$ is a
    $\mathbb Z_2$-$\mathrm{AE}$;
    \item if $X$ is an $\mathrm{ANR}$, then $(X,\alpha)$ is a
    $\mathbb Z_2$-$\mathrm{AR}$.
\end{enumerate}
\end{corollary}

\begin{proof}
By Lemma~\ref{l:locally convex}, $X$ is locally convex and Dini, and
hence countably Dini. Moreover,
Remark~\ref{r:sigma-monotonically-complete} implies that $X$ is
$\sigma$-monotonically conditionally complete. 
By Theorem~\ref{t:order-preserving mean implies equivariant mean}, $X$ admits an equivariant mean and if $X$ is an $\mathrm{AE}$, then it is a $\mathbb Z_2$-$\mathrm{AE}$. 
On the other hand, if $X$ is an $\mathrm{ANR}$, by Corollary~\ref{c:mejora ANR implica G-AR}, $X$ is a $\mathbb{Z}_2$-$\mathrm{AR}$. 
\end{proof}

\section{Applications}

\subsection{Decreasing involutions on the space of convex bodies of $\mathbb R^n$} \label{s:decreasing-involutions-convex-bodies}

Throughout this subsection, let $n\geq 2$. We denote by
$\mathcal K_{(0),b}^n$ the family of all compact convex subsets of
$\mathbb R^n$ that contain the origin in their interior. We equip this
space with the Hausdorff metric
\[
d_H(K,L)
=
\max\left\{
\sup_{x\in K}d(x,L),
\sup_{y\in L}d(y,K)
\right\}.
\]

Recall that the polar body of $K\in\mathcal K_{(0),b}^n$ is defined by
\[
K^\circ
=
\left\{
x\in\mathbb R^n:
\sup_{y\in K}\langle x,y\rangle\leq 1
\right\}.
\]
The polar mapping $K\longmapsto K^\circ$ is a continuous involution on $\mathcal K_{(0),b}^n$ and is decreasing
with respect to inclusion.

Consider the lattice operations
\[
K\wedge L:=K\cap L
\qquad\text{and}\qquad
K\vee L:=\operatorname{conv}(K\cup L).
\]
The join operation is continuous with respect to the Hausdorff metric
(see \cite[Proposition~2.4]{costantinikubis} and
\cite[Lemma~2.1]{sakaiyaguchi}). Since the polar mapping is a
decreasing involution, Lemma~\ref{l:semilattice implies lattice}
implies that the meet operation is also continuous. Thus, $\left(
\mathcal K_{(0),b}^n,
\subseteq,
\wedge,
\vee
\right)$ is a topological lattice.

We next stablish the topological properties of this lattice that will be
needed below. By \cite[Proposition~3.3(1)]{HiguerasJonard},
$\mathcal K_{(0),b}^n$ is an open subset of
$\mathcal K_0^n$ when both spaces are equipped with the
Attouch--Wets topology\footnote{For the precise definition of the Attouch--Wets metric and
the topology induced by it, see
\cite[Chapter~3, Section~3.1]{beer}.}. Moreover, \cite[Theorem 3.4]{HiguerasJonard} shows that
$\mathcal K_0^n$ is homeomorphic to the Hilbert cube. Since the
Attouch--Wets and Hausdorff topologies coincide on
$\mathcal K_{(0),b}^n$
(see \cite[Theorem~3.2]{sakaiyaguchi}), it follows that
$\mathcal K_{(0),b}^n$, equipped with the Hausdorff metric, is locally
compact and is an $\mathrm{ANR}$.

In addition, $\mathcal K_{(0),b}^n$ is contractible. Indeed, if
$\mathbb B$ denotes the closed Euclidean unit ball, the map
\[
F:
\mathcal K_{(0),b}^n\times[0,1]
\longrightarrow
\mathcal K_{(0),b}^n,
\qquad
F(K,t)=(1-t)K+t\mathbb B,
\]
is a contraction of $\mathcal K_{(0),b}^n$ to $\mathbb B$.
Consequently, $\mathcal K_{(0),b}^n$ is an $\mathrm{AR}$ and, in particular, is
connected.

We also claim that every order interval in
$\mathcal K_{(0),b}^n$ is compact. Let
$K,L\in\mathcal K_{(0),b}^n$ with $K\subseteq L$, and consider
\[
[K;L]
=
\left\{
C\in\mathcal K_{(0),b}^n:
K\subseteq C\subseteq L
\right\}.
\]
Let $(C_m)_{m\geq 1}$ be a sequence in $[K;L]$. Since every $C_m$ is
contained in the compact set $L$, the Blaschke selection theorem
\cite[Theorem~1.8.7]{Schneider} yields a subsequence
$(C_{m_j})_{j\geq 1}$ converging in the Hausdorff metric to a nonempty
compact convex set $C\subseteq L$. Moreover, since $K\subseteq C_{m_j}\subseteq L$ for every $j$, passing to the limit gives
\[
K\subseteq C\subseteq L.
\]
In particular, $C$ contains the origin in its interior, because
$K\subseteq C$. Hence $C\in[K;L]$. Thus $[K;L]$ is sequentially
compact and, since it is metrizable, it is compact.

Now, consider the Minkowski arithmetic mean
\[
A:
\mathcal K_{(0),b}^n\times
\mathcal K_{(0),b}^n
\longrightarrow
\mathcal K_{(0),b}^n,
\qquad
A(K,L)=\frac{K+L}{2}.
\]

For each $C\in\mathcal K_{(0),b}^n$, let
\[
h_C:\mathbb R^n\longrightarrow\mathbb R,
\qquad
h_C(u)=\sup_{x\in C}\langle x,u\rangle,
\]
denote the support function of $C$.

We claim that $A$ is a strict order-preserving mean. Suppose that
$K\subsetneq L$. Since Minkowski addition preserves inclusion and
$K$ and $L$ are convex, we have
\[
K=\frac{K+K}{2}
\subseteq
\frac{K+L}{2}
\subseteq
\frac{L+L}{2}=L.
\]

Moreover, since $K$ is properly contained in $L$, there exists
$u\in\mathbb S^{n-1}$ such that $h_K(u)<h_L(u)$ (see, e.g., \cite[Theorem~1.7.1]{Schneider}). Since the support
function of $\frac{K+L}{2}$ is given by
\[
h_{\frac{K+L}{2}}(x)
=
\frac{h_K(x)+h_L(x)}{2},
\]
the choice of $u$ guarantees that
\[
\begin{aligned}
h_K(u)
&=\frac{h_K(u)+h_K(u)}{2}\\
&<\frac{h_K(u)+h_L(u)}{2}
=h_{\frac{K+L}{2}}(u)\\
&<\frac{h_L(u)+h_L(u)}{2}
=h_L(u).
\end{aligned}
\]
Thus,
\[
K\neq \frac{K+L}{2}\neq L.
\]
Combining this with the inclusions above, we obtain
\[
K\subsetneq A(K,L)\subsetneq L.
\]
Hence, $A$ is a strict order-preserving mean and, in particular, an
$\mathrm{RS}$-order-preserving mean.

Now, let $\alpha:
\mathcal K_{(0),b}^n
\longrightarrow
\mathcal K_{(0),b}^n$ be an arbitrary involution that is decreasing with respect to
inclusion. By the De Morgan laws in
Lemma~\ref{l:decreasinglattices},
\[
\alpha(K\cap L)
=
\operatorname{conv}\bigl(\alpha(K)\cup\alpha(L)\bigr)
\]
for every $K,L\in\mathcal K_{(0),b}^n$. The representation result recalled in the introduction concerns the
larger space $\mathcal K_0^n$. For the present space, we use the
corresponding result of Böröczky and Schneider. By
\cite[Corollary]{BoroczkySchneider}, there exists a self-adjoint linear
isomorphism $T:\mathbb R^n\longrightarrow\mathbb R^n$
such that
\[
\alpha(K)=T(K^\circ),\quad\text{ for every  }K\in\mathcal K_{(0),b}^n.
\]

Next, we show that $A$ is an arithmetic mean with respect to
$\alpha$. Let
\[
H(K,L)
=
\left(
\frac{K^\circ+L^\circ}{2}
\right)^\circ
\]
be the harmonic mean associated with the polar involution. We claim
that
\[
A^\alpha(K,L)=H(K,L)
\]
for every $K,L\in\mathcal K_{(0),b}^n$. Indeed, let $T^*$ denote the adjoint of $T$ with respect to the
standard inner product on $\mathbb R^n$. By \cite[Chapter~IV, \S 2]{SchaeferWolff},
\[
T(C^\circ)
=
\left((T^*)^{-1}(C)\right)^\circ
\]
for every convex body $C$.  Set
\[
M:=\frac{K^\circ+L^\circ}{2}.
\]
Since $T$ is linear, it preserves Minkowski addition and positive
scalar multiplication. Hence,
\[
\frac{T(K^\circ)+T(L^\circ)}{2}
=
T\left(\frac{K^\circ+L^\circ}{2}\right)
=T(M).
\]
Therefore,
\begin{align*}
A^\alpha(K,L)
&=\alpha\bigl(A(\alpha(K),\alpha(L))\bigr)\\
&=
T\left(
\left[
\frac{T(K^\circ)+T(L^\circ)}{2}
\right]^\circ
\right)\\
&=T\bigl((T(M))^\circ\bigr)\\
&=
\left(
(T^*)^{-1}(T(M))
\right)^\circ\\
&=M^\circ\\
&=
\left(
\frac{K^\circ+L^\circ}{2}
\right)^\circ\\
&=H(K,L).
\end{align*}
By a result of Firey \cite{Firey} (see also
\cite[Theorem~2]{MilmanRotem}), the arithmetic and harmonic means
satisfy
\[
H(K,L)
=
\left(\frac{K^\circ+L^\circ}{2}\right)^\circ
\subseteq
\frac{K+L}{2}
=
A(K,L).
\]
It follows that
\[
A^\alpha(K,L)=H(K,L)\subseteq A(K,L),
\]
so $A$ is arithmetic with respect to $\alpha$.

Combining the preceding observations, we conclude that
$\mathcal K_{(0),b}^n$ is an $\mathrm{AR}$ and a locally compact and connected
topological lattice whose order intervals are compact. Moreover, it
admits an $\mathrm{RS}$-order-preserving arithmetic mean with respect to
$\alpha$. Corollary~\ref{c:locally-compact-lattice-equivariant-mean}
therefore applies and yields the following result.

\begin{theorem}\label{t:decreasing-involutions-convex-bodies}
Let $n\geq 2$ and let
\[
\alpha:
\mathcal K_{(0),b}^n
\longrightarrow
\mathcal K_{(0),b}^n
\]
be an involution that is decreasing with respect to inclusion. Then
$\mathcal K_{(0),b}^n$ admits an equivariant mean with respect to
$\alpha$. Moreover, $\bigl(\mathcal K_{(0),b}^n,\alpha\bigr)$
is a $\mathbb Z_2$-$\mathrm{AR}$.
\end{theorem}

\subsection{Convex functions and Legendre--Fenchel conjugation}
\label{ss:convex-functions-Legendre-involution}

Throughout this subsection, let $n\geq 1$. A convex function
$f:\mathbb R^n\to\mathbb R$ is said to be \textit{supercoercive} if
\[
\lim_{\|x\|\to\infty}\frac{f(x)}{\|x\|}=+\infty.
\]
Functions satisfying this condition are also called \textit{$1$-coercive}. We denote by
\[
\operatorname{Conv}_{\mathrm{sc}}(\mathbb R^n;\mathbb R)
\]
the family of all convex supercoercive functions
$f:\mathbb R^n\to\mathbb R$.

Since every finite convex function on $\mathbb R^n$ is continuous
\cite[Corollary~10.1.1]{Rockafellar}, every element of
$\operatorname{Conv}_{\mathrm{sc}}(\mathbb R^n;\mathbb R)$ belongs to
$C(\mathbb R^n,\mathbb R)$.

For $f\in\operatorname{Conv}_{\mathrm{sc}}(\mathbb R^n;\mathbb R)$,
its Legendre--Fenchel conjugate is defined by
\[
f^*(x)
:=
\sup_{y\in\mathbb R^n}
\bigl\{\langle x,y\rangle-f(y)\bigr\},
\qquad x\in\mathbb R^n.
\]
Recall that every proper lower semicontinuous convex function $f$ satisfies that $f^{*}$ is convex and $f^{**}=f$ (see, for example,
\cite[Theorem~11.1]{RockafellarWets}).
Moreover, such a function is supercoercive if and only if its conjugate
is finite everywhere
(see, e.g., \cite[Chapter~E, Section~1, Remark~1.3.10]
{HiriartUrrutyLemarechal}).

It follows that $f^*\in\operatorname{Conv}_{\mathrm{sc}}(\mathbb R^n;\mathbb R) $
whenever
$f\in\operatorname{Conv}_{\mathrm{sc}}(\mathbb R^n;\mathbb R)$.
Indeed, the supercoercivity of $f$ implies that
$f^{*}$ is finite, while $f^{**}=f$ being finite implies that $f^*$ is supercoercive. Consequently, the
Legendre--Fenchel transform
\[
\mathcal L:
\operatorname{Conv}_{\mathrm{sc}}(\mathbb R^n;\mathbb R)
\longrightarrow
\operatorname{Conv}_{\mathrm{sc}}(\mathbb R^n;\mathbb R),
\qquad
\mathcal L(f):=f^*,
\]
is well defined and satisfies $\mathcal L^2=\operatorname{id}$.

We equip $\operatorname{Conv}_{\mathrm{sc}}(\mathbb R^n;\mathbb R)$ with the compact-open topology inherited from
$C(\mathbb R^n,\mathbb R)$. This topology coincides with the topology
of uniform convergence on compact subsets
\cite[Theorem~43.7]{willard}. Moreover, since $\mathbb R^n$
is hemicompact, $C(\mathbb R^n,\mathbb R)$ endowed with this topology is
metrizable \cite[Exercise~43G(1)]{willard}. Consequently,
$\operatorname{Conv}_{\mathrm{sc}}(\mathbb R^n;\mathbb R)$ is also
metrizable.

For sequences of finite convex functions, epi-convergence is equivalent
to uniform convergence on compact subsets
\cite[Theorem~7.17]{RockafellarWets}. Hence, a
sequence in
$\operatorname{Conv}_{\mathrm{sc}}(\mathbb R^n;\mathbb R)$
epi-converges if and only if it converges in the compact-open
topology.\footnote{
Recall that epi-convergence of lower semicontinuous functions is defined
in terms of the Painlev\'e--Kuratowski convergence of their epigraphs.
See \cite[Chapter~7, Section~B, Definition~7.1]{RockafellarWets} or
\cite[Chapter~5, Definition~5.3.1]{beer}; for further equivalent
descriptions and related hyperspace topologies, see also \cite{beer}.}

The Legendre--Fenchel transform preserves epi-convergence
\cite[Theorem~11.34]{RockafellarWets}.
Consequently, $\mathcal L$ is sequentially continuous in the
compact-open topology. Since this topology is metrizable, $\mathcal L$
is continuous. Therefore, $\mathcal L$ is an involution and, consequently,
$\bigl(
\operatorname{Conv}_{\mathrm{sc}}(\mathbb R^n;\mathbb R),
\mathcal L
\bigr)
$
is a $\mathbb Z_2$-space.

We consider on
$\operatorname{Conv}_{\mathrm{sc}}(\mathbb R^n;\mathbb R)$ the
pointwise order
\[
f\leq g
\quad\Longleftrightarrow\quad
f(x)\leq g(x)
\quad\text{for every }x\in\mathbb R^n.
\]
It is well known that the Legendre--Fenchel transform reverses the
pointwise order  (see, e.g., 
\cite[Chapter~E, Section~1, Proposition~1.3.1-(vii)]
{HiriartUrrutyLemarechal}). 
Therefore,
\[
\bigl(
\operatorname{Conv}_{\mathrm{sc}}(\mathbb R^n;\mathbb R),
\mathcal L,\leq
\bigr)
\]
is a $\mathbb Z_2$-decreasing poset.

We next define a join operation on
$\operatorname{Conv}_{\mathrm{sc}}(\mathbb R^n;\mathbb R)$ by
\[
(f\vee g)(x):=\max\{f(x),g(x)\},
\qquad x\in\mathbb R^n.
\]
The function $f\vee g$ is convex
(\cite[Section~5, Theorem~5.5]{Rockafellar}) and finite everywhere.
Moreover, it is supercoercive, since $f\vee g\geq f$. Thus,
\[
f\vee g\in
\operatorname{Conv}_{\mathrm{sc}}(\mathbb R^n;\mathbb R).
\]

The operation $\vee$ is continuous. Indeed, since the space is metrizable, it
suffices to verify sequential continuity. Suppose that
$f_k\to f$ and $g_k\to g$ uniformly on compact subsets of
$\mathbb R^n$. Recall that, for any $a,b,c,d\in\mathbb R$,
\[
\left|\max\{a,b\}-\max\{c,d\}\right|
\leq
\max\{|a-c|,|b-d|\}.
\]
Therefore, for every compact subset $K\subseteq\mathbb R^n$ and $\varepsilon>0$, we can find  $N\in \mathbb{N}$ such that $\left|f_{k}\left(x\right)-f\left(x\right)\right|<\varepsilon$ and $\left|g_{k}\left(x\right)-g\left(x\right)\right|<\varepsilon$ for any $x\in K$ and $k\geq N$. Hence, 
\[
\begin{split}
\left|(f_{k}\vee g_{k})(x)-(f\vee g)(x)\right|&=\left|\max\left\{f_{k}\left(x\right),g_{k}\left(x\right)\right\}-\max\left\{f\left(x\right),g\left(x\right)\right\}\right|\\
& \leq \max\left\{\left|f_{k}\left(x\right)-f\left(x\right)\right|,\left|g_{k}\left(x\right)-g\left(x\right)\right|\right\}< \varepsilon.
\end{split}
\] 
We can now conclude that 
$f_k\vee g_k\to f\vee g$ uniformly on every compact subset of
$\mathbb R^n$, which proves that $\vee$ is continuous.

Consequently,
$\bigl(
\operatorname{Conv}_{\mathrm{sc}}(\mathbb R^n;\mathbb R),
\mathcal L,\leq,\vee
\bigr)
$
is a $\mathbb Z_2$-decreasing topological semilattice. By
Lemma~\ref{l:semilattice implies lattice}, it becomes a
$\mathbb Z_2$-decreasing lattice when endowed with the meet operation
\[
f\wedge g:=
\mathcal L\bigl(\mathcal L(f)\vee\mathcal L(g)\bigr)
=
(f^*\vee g^*)^*.
\]

\begin{remark}
In terms of epigraphs, the pointwise order corresponds to reverse
inclusion, while the join operation corresponds to intersection:
\[
f\leq g
\quad\Longleftrightarrow\quad
\operatorname{epi}g\subseteq\operatorname{epi}f,
\qquad
\operatorname{epi}(f\vee g)
=
\operatorname{epi}f\cap\operatorname{epi}g.
\]
\end{remark}

We first verify local convexity with respect to the pointwise order.

\begin{claim}
The lattice $\bigl(
\operatorname{Conv}_{\mathrm{sc}}(\mathbb R^n;\mathbb R),
\leq,\wedge,\vee
\bigr)
$ is locally convex with respect to $\leq$.
\end{claim}

\begin{proof}
Since the topology under consideration is the topology of uniform
convergence on compact subsets, a neighborhood basis at
$f\in\operatorname{Conv}_{\mathrm{sc}}(\mathbb R^n;\mathbb R)$ is
given by the sets
\[
U(f;K,\varepsilon)
=
\left\{
g\in\operatorname{Conv}_{\mathrm{sc}}(\mathbb R^n;\mathbb R):
\sup_{x\in K}|g(x)-f(x)|<\varepsilon
\right\},
\]
where $K\subseteq\mathbb R^n$ is compact and $\varepsilon>0$.

We claim that every such neighborhood is convex with respect to the
pointwise order. Suppose that $g,j\in U(f;K,\varepsilon)$ and that
$g\leq h\leq j$ for some
$h\in\operatorname{Conv}_{\mathrm{sc}}(\mathbb R^n;\mathbb R)$. Set
\[
\delta
=
\max\left\{
\sup_{x\in K}|g(x)-f(x)|,
\sup_{x\in K}|j(x)-f(x)|
\right\}.
\]
Since $g,j\in U(f;K,\varepsilon)$, we have $\delta<\varepsilon$.
Moreover, for every $x\in K$,
\[
f(x)-\delta
\leq g(x)
\leq h(x)
\leq j(x)
\leq f(x)+\delta.
\]
It follows that
\[
\sup_{x\in K}|h(x)-f(x)|
\leq\delta<\varepsilon.
\]
Thus, $h\in U(f;K,\varepsilon)$. Hence every member of the neighborhood
basis above is convex with respect to $\leq$, and the lattice is locally
convex.
\end{proof}

\begin{claim}
The lattice $\bigl(
\operatorname{Conv}_{\mathrm{sc}}(\mathbb R^n;\mathbb R),
\leq,\wedge,\vee
\bigr)$ is $\sigma$-monotonically conditionally complete.
\end{claim}

\begin{proof}
Let $(f_k)_{k\geq 1}$ be a monotonically increasing sequence in
$\operatorname{Conv}_{\mathrm{sc}}(\mathbb R^n;\mathbb R)$ that is
bounded from above by some
$u\in\operatorname{Conv}_{\mathrm{sc}}(\mathbb R^n;\mathbb R)$.
Define
\[
f(x)=\sup_{k\geq 1}f_k(x),
\qquad x\in\mathbb R^n.
\]
Since
\[
f_1(x)\leq f(x)\leq u(x)
\]
for every $x\in\mathbb R^n$, the function $f$ is finite everywhere.
Moreover, as a pointwise supremum of convex functions, $f$ is convex
(\cite[Theorem~5.5]{Rockafellar}). It is also supercoercive,
because $f_1\leq f$ and $f_1$ is supercoercive. Therefore,
\[
f\in\operatorname{Conv}_{\mathrm{sc}}(\mathbb R^n;\mathbb R).
\]

By its definition, $f$ is an upper bound of $(f_k)_{k\geq 1}$.
Furthermore, if
$v\in\operatorname{Conv}_{\mathrm{sc}}(\mathbb R^n;\mathbb R)$
is any other upper bound of this sequence, then
\[
f_k(x)\leq v(x)
\]
for every $k\geq1$ and every $x\in\mathbb R^n$. Taking the supremum
over $k$ gives $f\leq v$. Hence, $f=\sup_{k\geq1}f_k.$

Thus, every monotonically increasing sequence that is bounded from
above has a supremum. Since $\mathcal L$ is a decreasing involution,
Proposition~\ref{p:duality-monotone-sequences}-(2) shows that every monotonically
decreasing sequence bounded from below also has an infimum. Therefore,
the lattice is $\sigma$-monotonically conditionally complete.
\end{proof}

\begin{claim}
The lattice $\bigl(
\operatorname{Conv}_{\mathrm{sc}}(\mathbb R^n;\mathbb R),
\leq,\wedge,\vee
\bigr)$
is countably Dini.
\end{claim}

\begin{proof}
Let $(f_k)_{k\geq1}$ be a monotonically increasing sequence in
$\operatorname{Conv}_{\mathrm{sc}}(\mathbb R^n;\mathbb R)$ having a
supremum
$f\in\operatorname{Conv}_{\mathrm{sc}}(\mathbb R^n;\mathbb R)$.
By the proof of the preceding claim, $f(x)=\sup_{k\geq1}f_k(x)$
for every $x\in\mathbb R^n$, and therefore $f_k(x)$ converges to $f(x)$ pointwise.

Since $f_k$ and $f$ are finite convex functions, it follows from
\cite[Theorem~10.8]{Rockafellar} that $(f_k)$ converges
uniformly to $f$ on every compact subset of $\mathbb R^n$. Thus,
$f_k\to f$ in the compact-open topology, and the lattice is countably
Dini.
\end{proof}

\begin{claim}
The space $\operatorname{Conv}_{\mathrm{sc}}(\mathbb R^n;\mathbb R)$ is an $\mathrm{AR}$.
\end{claim}

\begin{proof}
The space $C(\mathbb R^n,\mathbb R)$, endowed with the compact-open
topology, is a locally convex topological vector space
(see, e.g., \cite[Chapter~6, Remark~1]{sakai}). We claim that
$\operatorname{Conv}_{\mathrm{sc}}(\mathbb R^n;\mathbb R)$ is a convex
subset of this space.

Indeed, let
$f,g\in\operatorname{Conv}_{\mathrm{sc}}(\mathbb R^n;\mathbb R)$ and
let $t\in[0,1]$. The function
\[
(1-t)f+tg
\]
is finite and convex
(\cite[Theorem~5.2]{Rockafellar}). If $0<t<1$, then
\[
\frac{((1-t)f+tg)(x)}{\|x\|}
=
(1-t)\frac{f(x)}{\|x\|}
+
t\frac{g(x)}{\|x\|}
\longrightarrow+\infty
\]
as $\|x\|\to\infty$. Thus, $(1-t)f+tg$ is supercoercive. The cases
$t=0$ and $t=1$ are immediate, and hence $(1-t)f+tg
\in
\operatorname{Conv}_{\mathrm{sc}}(\mathbb R^n;\mathbb R)$
for every $t\in[0,1]$.

Therefore,
$\operatorname{Conv}_{\mathrm{sc}}(\mathbb R^n;\mathbb R)$ is a
metrizable convex subset of a locally convex topological vector space.
Thus, by Dugundji's theorem (\cite[Theorem~6.1.1]{sakai}), the space $\operatorname{Conv}_{\mathrm{sc}}(\mathbb R^n;\mathbb R)$
is an $\mathrm{AR}$, as desired. 
\end{proof}

\begin{claim}
The map
\[
\mathcal A:
\operatorname{Conv}_{\mathrm{sc}}(\mathbb R^n;\mathbb R)
\times
\operatorname{Conv}_{\mathrm{sc}}(\mathbb R^n;\mathbb R)
\longrightarrow
\operatorname{Conv}_{\mathrm{sc}}(\mathbb R^n;\mathbb R),\]
given by 
$\mathcal A(f,g)=\frac{f+g}{2},$ is a strict order-preserving arithmetic mean with respect to
$\mathcal L$.
\end{claim}

\begin{proof}
The map $\mathcal A$ is well-defined because
$\operatorname{Conv}_{\mathrm{sc}}(\mathbb R^n;\mathbb R)$ is a convex
subset of $C(\mathbb R^n,\mathbb R)$. It is continuous because
$C(\mathbb R^n,\mathbb R)$, endowed with the compact-open topology, is
a topological vector space. Moreover,
\[
\mathcal A(f,g)=\mathcal A(g,f)
\qquad\text{and}\qquad
\mathcal A(f,f)=f,
\]
so $\mathcal A$ is a mean.

Next, we verify that it is strict order preserving. Suppose that $f<g$.
Then
\[
f(x)
\leq
\frac{f(x)+g(x)}{2}
\leq
g(x)
\]
for every $x\in\mathbb R^n$. Moreover, there exists
$y\in\mathbb R^n$ such that $f(y)<g(y)$, and hence
\[
f(y)
<
\frac{f(y)+g(y)}{2}
<
g(y).
\]
Therefore, $f<\mathcal A(f,g)<g.$

It remains to prove that $\mathcal A$ is arithmetic with respect to
$\mathcal L$. For every
$u,v\in\operatorname{Conv}_{\mathrm{sc}}(\mathbb R^n;\mathbb R)$ and
$x\in\mathbb R^n$, we have
\[
\begin{aligned}
\mathcal L\bigl(\mathcal A(u,v)\bigr)(x)
&=
\sup_{y\in\mathbb R^n}
\left\{
\langle x,y\rangle-\frac{u(y)+v(y)}{2}
\right\}                                                   \\
&=
\sup_{y\in\mathbb R^n}
\left\{
\frac{\langle x,y\rangle-u(y)}{2}
+
\frac{\langle x,y\rangle-v(y)}{2}
\right\}                                                   \\
&\leq
\frac{\mathcal L(u)(x)+\mathcal L(v)(x)}{2}                \\
&=
\mathcal A\bigl(\mathcal L(u),\mathcal L(v)\bigr)(x).
\end{aligned}
\]
Applying this inequality to $u=\mathcal L(f)$ and
$v=\mathcal L(g)$, and using $\mathcal L^2=\operatorname{id}$, we
obtain
\[
\begin{aligned}
\mathcal A^{\mathcal L}(f,g)
&=
\mathcal L\bigl(
\mathcal A(\mathcal L(f),\mathcal L(g))
\bigr)                                                     \\
&\leq
\mathcal A\bigl(\mathcal L^2(f),\mathcal L^2(g)\bigr)      \\
&=
\mathcal A(f,g).
\end{aligned}
\]
Thus, $\mathcal A$ is an arithmetic mean with respect to $\mathcal L$.
\end{proof}

Combining the preceding claims with
Theorem~\ref{t:order-preserving mean implies equivariant mean}, we
obtain the following result.

\begin{theorem}\label{t:geometric-mean-convex-functions}
The space
$\operatorname{Conv}_{\mathrm{sc}}(\mathbb R^n;\mathbb R)$ admits an
$\mathcal L$-equivariant mean
\[
\mathcal G:
\operatorname{Conv}_{\mathrm{sc}}(\mathbb R^n;\mathbb R)
\times
\operatorname{Conv}_{\mathrm{sc}}(\mathbb R^n;\mathbb R)
\longrightarrow
\operatorname{Conv}_{\mathrm{sc}}(\mathbb R^n;\mathbb R),
\]
which we call the geometric mean associated with $\mathcal A$ and
$\mathcal L$. In particular,
\[
\mathcal G(f^*,g^*)
=
\mathcal G(f,g)^*
\qquad
\text{for every }
f,g\in\operatorname{Conv}_{\mathrm{sc}}(\mathbb R^n;\mathbb R).
\]
Moreover, $\bigl(
\operatorname{Conv}_{\mathrm{sc}}(\mathbb R^n;\mathbb R),
\mathcal L
\bigr)$ is a $\mathbb Z_2$-$\mathrm{AR}$.
\end{theorem}

\begin{remark}
By Lemma~\ref{l:properties-dual-mean}, the dual mean
\[
\mathcal A^{\mathcal L}(f,g)
=
\mathcal L\bigl(
\mathcal A(\mathcal L(f),\mathcal L(g))
\bigr)
\]
is a strict order-preserving harmonic mean with respect to
$\mathcal L$. This operation can be expressed in terms of the infimal
convolution.

Recall that the \textit{infimal convolution} of two functions
$f,g:\mathbb R^n\to\mathbb R\cup\{+\infty\}$ is defined by the rule
\[
(f\square g)(x):
=
\inf\left\{
f(x_1)+g(x_2):
x_1+x_2=x
\right\}
\]
(see, e.g., \cite[Chapter~B, Section~2.3]{HiriartUrrutyLemarechal}).
For $t>0$, the dilation of $f$ by $t$ is the function $f_t$ defined by
\[
f_t(x):
=
t f\left(\frac{x}{t}\right),
\qquad x\in\mathbb R^n
\]
(see \cite[Chapter~B, Section~2.2]{HiriartUrrutyLemarechal}).

Now, let
$f,g\in\operatorname{Conv}_{\mathrm{sc}}(\mathbb R^n;\mathbb R)$.
Recall that $f^*$ and $g^*$ are finite everywhere. Therefore, the
functions $\frac12 f^*$ and $\frac12 g^*$ are finite. It follows from
\cite[Theorem~16.4]{Rockafellar} that
\[
\left(\frac12 f^*+\frac12 g^*\right)^*
=
\left(\frac12 f^*\right)^*
\square
\left(\frac12 g^*\right)^*.
\]

Applying the scaling identity
\cite[Chapter~E, Proposition~1.3.1-(ii)]
{HiriartUrrutyLemarechal} to $f^*$ and $g^*$ with
$t=\frac12$, and using $f^{**}=f$ and $g^{**}=g$, we obtain
\[
\left(\frac12 f^*\right)^*=f_{\frac12}
\qquad\text{and}\qquad
\left(\frac12 g^*\right)^*=g_{\frac12}.
\]
Consequently,
\[
\begin{aligned}
\mathcal A^{\mathcal L}(f,g)
&=
\mathcal L\bigl(
\mathcal A(\mathcal L(f),\mathcal L(g))
\bigr)\\
&=
\left(\frac12 f^*+\frac12 g^*\right)^*\\
&=
f_{\frac12}\square g_{\frac12}.
\end{aligned}
\]
Equivalently,
\[
\mathcal A^{\mathcal L}(f,g)(x)
=
\inf_{y\in\mathbb R^n}
\left\{
\frac12 f(2y)
+
\frac12 g\bigl(2(x-y)\bigr)
\right\},
\qquad x\in\mathbb R^n.
\]
\end{remark}

\end{document}